\documentclass{article}
\usepackage{graphicx, amscd, amsmath,amssymb, amstext,amsthm,amsfonts, enumitem, comment, breqn}
\usepackage{mathrsfs}
\usepackage[utf8]{inputenc}
\usepackage[T1]{fontenc}
\usepackage{hyperref}
\usepackage[a4paper, left=2.6cm, right=2.6cm, top=2.5cm, bottom=2.5cm]{geometry}
\usepackage{authblk}
\usepackage{marvosym}
\usepackage{tikz}
\usetikzlibrary{positioning, arrows.meta}
\usepackage[pagewise]{lineno}

\newcommand{\al} {\alpha}       

\newcommand{\be} {\beta}

\newcommand{\Ga}{\Gamma}
\newcommand{\de} {\delta}

\newcommand{\vep}{\varepsilon}

\newcommand{\la} {\lambda}

\newcommand{\si} {\sigma}

\newcommand{\SB}{{\mathcal B}}

\newcommand{\SL}{{\mathcal L}}

\newcommand{\SU}{{\mathcal U}}
\newcommand{\SV}{{\mathcal V}}

\newcommand{\N}{\mathbb{N}}

\newcommand{\seqx}{\underline{x}}
\newcommand{\seqw}{\underline{w}}
\newcommand{\seqz}{\underline{z}}
\newcommand{\orbx}{\seqx_T}
\newcommand{\orbz}{\seqz_T}
\newcommand{\dbar}{\bar{d}}
\newcommand{\dist}{\rho}
\newcommand{\alf}{\mathscr{A}}
\newcommand{\FS}{\alf^\infty}
\newcommand{\emptyword}{\lambda}
\newcommand{\lang}{\mathcal{L}}
\newcommand{\distBT}{\dist^T_B}
\newcommand{\distB}{\dist_B}

\newtheorem{mainthm}{Theorem}

\newtheorem*{cor}{Corollary}

\makeatletter
\def\keywords{\xdef\@thefnmark{}\@footnotetext}
\makeatother

\newtheorem{theorem}{Theorem}[section] 
\newtheorem{lemma}[theorem]{Lemma}     
\newtheorem{corollary}[theorem]{Corollary}
\newtheorem{prop}[theorem]{Proposition}
\newtheorem{remark}[theorem]{Remark}
\newtheorem{question}[theorem]{Question}

\theoremstyle{definition}
\newtheorem{defi}[theorem]{Definition}

\newtheorem{fact}[theorem]{Fact}
\newtheorem{example}[theorem]{Example}

\title{On the weakness of the vague specification property}

\author{Melih Emin Can\textsuperscript{1 \href{mailto:melih.can@im.uj.edu.pl}{\Letter}} and Alexandre Trilles\textsuperscript{1,2 \href{mailto:alexandre.trilles@doctoral.uj.edu.pl}{\Letter}}}

\affil{\textsuperscript{1}Faculty of Mathematics and Computer Science, Jagiellonian University in Krak\'ow \\ Ul. {\L}ojasiewicza 6, 30-348 Krak\'ow, Poland}

\affil{\textsuperscript{2}Doctoral School of Exact and Natural Sciences, Jagiellonian University in Krak\'ow \\ Ul. {\L}ojasiewicza 11, 30-348 Krak\'ow, Poland}

\date{}

\begin{document}

	\maketitle

    \begin{abstract}
        We show that the vague specification property is strictly weaker than most of the specification-like properties, by establishing its equivalence with the asymptotic average shadowing property. In particular, we see that the weak specification property implies the vague specification property, but the converse does not hold, answering the question posed by Downarowicz and Weiss in [\emph{Ergod. Th. \& Dynam. Sys.} \textbf{44}(9) (2024), 2565-2580]. Additionally, we prove that, for surjective systems, the asymptotic average shadowing property is equivalent to the average shadowing property if the phase space is complete with respect to the dynamical Besicovitch pseudometric. We use the combination of both results to prove that the proximal and minimal shift spaces from [\emph{Ergod. Th. \& Dynam. Sys.}, \textbf{45}(2) (2025), 396–426] possess the vague specification property (asymptotic average shadowing property). 
        Our findings also allow us to address a couple of questions from [\emph{Fund. Math.}, \textbf{224}(3) (2014), 241–278] about the asymptotic average shadowing property.
    \end{abstract}
    
	\section{Introduction}
 
\keywords{\emph{Mathematics Subject Classification 2020.} Primary 37B65 ; Secondary 37B05, 37B10}

\keywords{\emph{Key words.} shadowing property, specification property, Besicovitch pseudometric, shift space, minimal system, proximal system.}

\keywords{MEC was supported by NCN Polonez Bis grant no. 2021/43/P/ST1/02885 and NCN Polonez Bis grant No. 2022/47/P/ST1/00854.
AT was supported by NCN Polonez Bis grant no. 2021/43/P/ST1/02885 and NCN Sonata Bis grant no. 2019/34/E/ST1/00082.}
       
    In the early 1970s, Bowen \cite{Bowen} introduced the specification property and used it to study Axiom A diffeomorphisms. Roughly speaking, the specification property guarantees the existence of an orbit to trace arbitrary, but finite, collections of segments of orbits, provided that these segments are sparse enough in time. 
    Systems with the specification property have very rich dynamics since this property implies, for example, topological mixing, positive topological entropy, and intrinsic ergodicity (see \cite{Bowen_MME, Sigmund}).

    Although the specification property is satisfied by important classes of systems, such as mixing interval maps and mixing shifts of finite type, weaker forms of specification were introduced in order to cover relevant classes without specification in the sense of Bowen. Some examples of these variants are the weak specification property \cite{Marcus_weak_spec,Dateyama_2}, satisfied by all automorphisms of compact Abelian groups for which the Haar measure is ergodic \cite{Dateyama_automorphisms} and the almost specification \cite{Pfister_Sullivan, Thompson}, satisfied by $\be$-shifts.
    Specification-like properties have been studied by several authors, and we refer to \cite{KLO} for an overview. Below, we present a simplified version of the diagram from \cite{KLO} and emphasize that there are no additional implications between these notions besides those given by transitivity.

    \begin{center}
    \begin{tikzpicture}[
    node distance=.7cm and 1.3cm,
    every node/.style={draw, rounded corners, minimum width=3cm, align=center},
    every path/.style={thick, -Implies}
    ]

    \node (spec) {specification};
    \node (periodic) [left=of spec, yshift=-.8cm] {periodic \\specification};
 
    \node (almost) [right= of spec, xshift=2mm] {almost\\specification};
    \node (weak) [right=of spec, yshift=-1.5cm, xshift=2mm] {weak \\specification};
    \node (periodicweak) [left=of weak, xshift=-.2cm] {periodic\\ weak specification};

    \draw (spec.east) edge[thick, double, -Implies] (almost.west);
    \draw (periodic) edge[thick, double, -Implies] (spec.west); 
    \draw (periodic) edge[thick, double, -Implies] (periodicweak.west);
    \draw (periodicweak) edge[thick, double, -Implies] (weak);
    \draw (spec.south east) edge[double, -Implies] (weak.north west);
    \end{tikzpicture}
    \end{center}

    The specification property and its variants have been extensively used in ergodic theory to investigate, for example, measures of maximal entropy, large deviation, and the existence of generic points for not necessarily ergodic measures. Recently, Downarowicz and Weiss \cite{DW} used the weak specification to prove results about lifting generic points. Kamae \cite{Kamae} obtained similar results in 1976 under the assumption of a less explored specification-like property (not considered in \cite{KLO}), called the vague specification property, which he introduced in the same work. This connection led to the question posed at the end of Downarowicz and Weiss' work about how the weak and the vague specification properties relate to each other.
    
    We show that the vague specification property is implied by all the properties displayed in the diagram above, providing, as a consequence, an answer to Downarowicz and Weiss' question. This follows from Theorem \ref{thm_AASP_VSP_intro}, where we establish the equivalence of the vague specification and a variant of the shadowing property introduced by Gu \cite{Gu}, called the asymptotic average shadowing property, combined with previous results on the latter property.

    \begin{mainthm}\label{thm_AASP_VSP_intro}
        A topological dynamical system $(X,T)$ has the vague specification property if and only if $(X,T)$ has the asymptotic average shadowing property.
    \end{mainthm}

    In \cite{KLO2}, it is shown that both the almost specification and the weak specification properties imply the asymptotic average shadowing property. 
    Therefore, we conclude that the vague specification property is one of the weakest forms of specification (see \cite{KLO, KLO2}) and we can complete the diagram above by adding the vague specification property in the following way.

    \begin{center}
    \begin{tikzpicture}[
    node distance=.7cm and 1.3cm,
    every node/.style={draw, rounded corners, minimum width=3cm, align=center},
    every path/.style={thick, -Implies}
    ]

    \node (spec) {specification};
    \node (periodic) [left=of spec, yshift=-.8cm] {periodic \\specification};
 
    \node (almost) [right= of spec, xshift=2mm] {almost\\specification};
    \node (weak) [right=of spec, yshift=-1.5cm, xshift=2mm] {weak \\specification};
    \node (periodicweak) [left=of weak, xshift=-.2cm] {periodic\\ weak specification};
    \node (vague) [right=of almost, yshift=-.7cm, xshift=-5mm] {vague\\specification};

    \draw (spec.east) edge[thick, double, -Implies] (almost.west);
    \draw (periodic) edge[thick, double, -Implies] (spec.west); 
    \draw (periodic) edge[thick, double, -Implies] (periodicweak.west);
    \draw (periodicweak) edge[thick, double, -Implies] (weak);
    \draw (spec.south east) edge[double, -Implies] (weak.north west);
    \draw (almost) edge[double, -Implies] (vague);
    \draw (weak) edge[double, -Implies] (vague);
    \end{tikzpicture}
    \end{center}

    Furthermore, while both the weak and the almost specification properties imply the existence of several pairwise disjoint invariant sets, we show that the vague specification property can be satisfied by proximal and minimal systems, which reinforces its weakness. 
    Indeed, we prove that the proximal example and the class of minimal shift spaces constructed in \cite{CKKK} possess the vague specification property.

    \begin{mainthm}\label{thm_minimal_proximal_intro}
        There are non-trivial minimal and non-trivial proximal systems which have the vague specification property. 
    \end{mainthm}

    Theorems \ref{thm_AASP_VSP_intro} and \ref{thm_minimal_proximal_intro} originated from our studies on variants of the classical shadowing property. The shadowing property, also known as pseudo-orbit tracing property, played a crucial role in the development of the theory of hyperbolic dynamical systems due to the Shadowing Lemma, which is a key step in the proof of structural stability of hyperbolic sets (see \cite{Bowen,Katok-Hasselblatt}). Informally, the shadowing property, means that every pseudo-orbit (a sequence of points in the phase space such that every point is close to the image of the previous one) is traced by an actual orbit.

    Besides its applications to differentiable dynamics, the shadowing property became a subject of studies in topological dynamics. This led to the introduction of different notions of pseudo-orbits and tracing, resulting in variants of the shadowing property, such as the limit shadowing property \cite{limit_shadowing}, the average shadowing property \cite{Blank}, and the asymptotic average shadowing property \cite{Gu}.
    
    While the original notion of pseudo-orbit consists of small errors at each iterate, in the asymptotic average shadowing property, the type of pseudo-orbits considered requires the average of these errors to vanish in the limit. Similarly, the notion of tracing is also defined by the average distance between the iterates of a point and the terms of the pseudo-orbits vanishing in the limit. The investigation of the connection between the asymptotic average shadowing and the vague specification properties, which led to Theorem \ref{thm_AASP_VSP_intro}, was motivated by the fact that the concept of tracing considered in both properties, although differently stated, is actually the same. The notion of tracing used on the vague specification property is defined via the asymptotic density of natural numbers.

    The average shadowing property, introduced by Blank in \cite{Blank}, and recently revisited by himself in \cite{Blank2}, was an inspiration for the asymptotic average shadowing property. Thus, when studying the asymptotic average shadowing property, it is natural to ask its relation with the average shadowing property.
    It is already known that the asymptotic average shadowing property implies the average shadowing property. This result was first proved for surjective topological dynamical systems by Kulczycki, Kwietniak, and Oprocha in \cite{KKO}. Later, the result was improved by Wu, Oprocha and Chen in \cite{WOC}, who showed that the assumption of surjectivity could be dropped. In \cite{KKO}, one of the questions posed by the authors ask whether the converse is true (Question 10.3 in \cite{KKO}).

    In Theorem \ref{thm_AASP_ASP_intro} we provide a step toward answering Question 10.3 in \cite{KKO}, by showing that for a surjective topological dynamical system $(X,T)$, the asymptotic average shadowing property is equivalent to $(X,T)$ being complete with respect to the dynamical Besicovitch pseudometric and having the average shadowing property.
    Roughly speaking, the dynamical Besicovitch pseudometric, defined on the phase space of a topological dynamical system, measures the asymptotic average distance between the iterates of two initial points. It is worth noting that some authors define the dynamical Besicovitch pseudometric without naming it, and we decided to adopt this terminology in accordance with \cite{KLO2}.

    \begin{mainthm}\label{thm_AASP_ASP_intro}
        Let $(X,T)$ be a surjective topological dynamical system. Then $(X,T)$ has the asymptotic average shadowing property if and only if $(X,T)$ is Besicovitch complete and has the average shadowing property.
    \end{mainthm}

    Although we believe that Theorem \ref{thm_AASP_ASP_intro} is interesting on its own, we also stress that this result is crucial for the proof of Theorem \ref{thm_minimal_proximal_intro}. In fact, we use an equivalent version of Theorem \ref{thm_AASP_ASP_intro} for shift spaces, that is, subsystems of the full shift.
    Several dynamical properties, such as topological mixing and the specification property, have simpler formulations in the context of shift spaces. 
    Konieczny, Kupsa, and Kwietniak \cite{KKK} introduced the $\dbar$-shadowing property, a variant of the shadowing property for shift spaces that is strongly related to the average shadowing property. Indeed, the average shadowing property implies the $\dbar$-shadowing property, and the converse holds for surjective shift spaces. These relations were claimed without a proof in \cite{KKK}, as we rely on this equivalence in our work, for the sake of completeness, we present a proof for these statements.

    \begin{mainthm}\label{thm_ASP_dbar_intro}
        Let $X$ be a surjective shift space. Then $X$ has the average shadowing property if and only if $X$ has the $\dbar$-shadowing property.
    \end{mainthm}
    
    We recall that the $\dbar$ and the Besicovitch pseudometrics are uniformly equivalent on shift spaces (see \cite{KLO2}). As a consequence, Theorem \ref{thm_AASP_ASP_intro} implies that, for surjective shift spaces, if we assume completeness with respect to the $\dbar$ pseudometric, then the $\dbar$-shadowing and the asymptotic average shadowing properties are equivalent. In particular, from Theorem \ref{thm_AASP_VSP_intro}, we obtain the following corollary.

    \begin{cor}
      Let $X$ be a surjective shift space. Then $X$ has the vague specification property if and only if $X$ is $\dbar$-complete and has the $\dbar$-shadowing property.
    \end{cor}
    
    In \cite{CKKK}, the authors constructed a proximal shift space and a class of minimal shift spaces with the $\dbar$-shadowing property. The authors acknowledged Piotr Oprocha for the ideas that led to the construction of the minimal shift spaces. We prove that these examples are $\dbar$-complete, and therefore, as they are surjective, they have the vague specification property which proves Theorem \ref{thm_minimal_proximal_intro}. 
    
    Specification-like properties are the most common tools used to prove entropy-density of ergodic measures in the space of invariant measures (see \cite{KLO, Pavlov, Sigmund} for definitions and details). The proximal example and the class of minimal shift spaces constructed in \cite{CKKK} are shown to have ergodic measures entropy-dense in the space of invariant measures. The surprise about such systems was that they do not seem to manifest any specification-like property and, therefore, different strategies are used to obtain the result. However, our results confirm that all these examples possess the vague specification property, which suggests that the vague specification property may imply entropy-density of ergodic measures.
    
    Lastly, we also use the equivalence of the vague specification and the asymptotic average shadowing properties to answer other questions about the latter property raised in \cite{KKO}. Indeed, the minimal shift space from Theorem \ref{thm_minimal_proximal_intro} answers Question 10.2 in \cite{KKO}, confirming that systems with the asymptotic average shadowing property may be minimal. Moreover, we can use Kamae's result (Proposition 5 in \cite{Kamae}), which shows that the vague specification property is inherited by factors, to address Question 10.6 in \cite{KKO}, confirming that the asymptotic average shadowing property is inherited by factors.

    Here is the organization of the paper. In Section \ref{section_preliminaries}, we present some basic notions and notation used throughout this work. In Section \ref{section_pseudo_orbits}, we analyze different notions of pseudo-orbits and the relations between them, and we prove the equivalence between asymptotic average pseudo-orbits and vague pseudo-orbits, which is the key step to prove Theorem \ref{thm_AASP_VSP_intro}. In Section \ref{section_tracing_properties}, we prove Theorems \ref{thm_AASP_VSP_intro} and \ref{thm_AASP_ASP_intro}. In Section \ref{section_shift_spaces}, we present the basic definitions of shift spaces, we study the $\dbar$ pseudometric and $\dbar$-shadowing property, and we prove Theorem \ref{thm_ASP_dbar_intro}. Section \ref{section_examples} is devoted to the proof of Theorem \ref{thm_minimal_proximal_intro}. Finally, in Section \ref{section_applications}, we discuss some consequences of our results. 
    
    \section{Preliminaries}\label{section_preliminaries}
	We write $\N = \{1,2,\ldots\}$ and $\N_0 = \N \cup \{0\}$. For $A \subseteq \N_0$ and $i \in \N_0$, we write $A-i = \{n \in \N_0 : n+i \in A\}$. For every $A \subset \N_0$, we denote by $|A| \in \N_0 \cup \{\infty\}$ the cardinality of $A$. For $n,m \in \N_0$ we write $[n,m) = \{i\in \N_0 : n \leq i < m \}$, $(n,m] = \{i\in \N_0 : n <i \leq m \}$  and $[n,m] = \{i\in \N_0 : n \leq i \leq m \}$.

    The \textbf{upper asymptotic density} of $A \subset \N_0$ is defined as
    $$\dbar(A) = \limsup_{n \to \infty} \frac{|A \cap [0,n)|}{n}.$$
 
    The \textbf{lower asymptotic density} of $A \subset \N_0$ is defined as
    $$\underline{d}(A) = \liminf_{n \to \infty} \frac{|A \cap [0,n)|}{n}.$$

    In case $\dbar(A) = \underline{d}(A)$, we define the \textbf{asymptotic density} of $A \subset \N_0$ by
    $$d(A) = \dbar(A) =\underline{d}(A).$$
    
    Throughout this work we assume that $X$ is a compact metrizable topological space and $\dist$ is a compatible metric on $X$. For simplicity, we assume that the diameter of $X$ with respect to $\dist$ is at most $1$. Our results do not depend on the choice of $\dist$.
	
    We denote by $X^{\infty}$ the family of all $X$-valued sequences indexed by $\N_0$, that is, 
    \[
        X^\infty = \left\{ \left\{x_n\right\}_{n=0}^\infty : x_n\in X \text{ for all }n \in \N_0\right\}.
    \]
    Typically, we write $\seqx = \{x_n\}_{n=0}^\infty$ for the elements of $X^{\infty}$. We denote by $S$ the shift operator acting on $X^\infty$, that is, $S(\seqx) = \{x_{n+1}\}_{n=0}^\infty$.
    We always endow $X^{\infty}$ with the product topology, which is compatible with the metric
    $$\dist^\infty(\seqx,\seqz) = \sum_{n=0}^\infty \frac{\dist(x_n,z_n)}{2^{n+1}}.$$

    Note that the diameter of $X^\infty$ with respect to $\dist^\infty$ is again at most $1$. Immediately from the definition of $\dist^\infty$ we obtain the following result.
    \begin{prop}
        Let $\vep >0$ and $k \in \N$. If $0<\de<\frac{\vep}{2^k}$ and $\dist^\infty(\seqx,\seqz) < \de$ for $\seqx, \seqz \in X$, then
        $$\max_{0\leq n <k} \dist(x_n, z_n) < \vep.$$
    \end{prop}
    
	A pair $(X,T)$ is said to be a \textbf{topological dynamical system} if $T \colon X \to X$ is a continuous map. The \textbf{orbit} of a point $x \in X$ under $T$ is usually defined as the set $\{T^n(x) \in X : n \in \N_0\}$ but, for our purposes, it is more convenient to treat orbits as sequences and we write $\orbx=\{T^n(x)\}_{n=0}^{\infty} \in X^\infty$. We write $\Ga_T = \{\seqx_T \in X^\infty : x \in X\}$ for the space of all orbits under $T$. It is not hard to see that $\Ga_T$ is a closed subspace of $X^\infty$, and therefore, $\Ga_T$ is compact.

    The \textbf{Besicovitch pseudometric} on $X^\infty$ is defined for $\seqx,\seqz \in X^\infty$ as
	$$\dist_B(\seqx, \seqz) = \limsup_{n \to \infty} \frac{1}{n}\sum_{i=0}^{n-1}\dist(x_i, z_i).$$

	Another pseudometric on $X^\infty$ is given for $\seqx,\seqz \in X^\infty$ by
 	\begin{align*}
		\dist_B'(\seqx, \seqz) &= \inf \left\{\vep >0 : \dbar(\{ n \in \N_0 : \dist(x_n, z_n) \geq \vep\}) < \vep\right\} \\ &=\inf \left\{\vep >0 : \underline{d}(\{ n \in \N_0 : \dist(x_n, z_n) < \vep\}) \geq 1 -\vep\right\}.
	\end{align*}

    \begin{remark}\label{equivalence_pseudometrics}
        It is easy to see that the pseudometrics $\dist_B$ and $\dist_B'$ on $X^\infty$ are uniformly equivalent (see Lemma 2 in \cite{KLO2} for a proof).
    \end{remark}
    From the definition of $\dist_B'$ we obtain the following equivalence.

    \begin{prop}\label{lem:eq-of-besicovitch-and-density}
		Let $\seqx, \seqz \in X^\infty$. Then $\dist_B' (\seqx, \seqz)=0$ if and only if for every $\vep >0$ we have
		$$d(\{n \in \N_0 : \dist(x_n, z_n)< \vep\})= 1.$$
	\end{prop}

    \begin{proof}
        Fix $\seqx, \seqz \in X^\infty$. Assume that for every $\vep>0$ we have $$d(\{n \in \N : \dist(x_n, z_n)< \vep\})= 1.$$ Then by definition $\dist_B'(\seqx, \seqz)=0$.
		
        To prove the converse, assume that $\dist_B'(\seqx, \seqz)=0$ and fix $\vep>0$. We observe that for each $0<\delta<\vep$ we have
		$$\{ n \in \N : \dist(x_n, z_n) < \delta\} \subseteq\{ n \in \N : \dist(x_n, z_n) < \vep\}.$$
        Fix $0<\delta<\vep$. Since $\dist_B' (\seqx, \seqz) < \delta$, one has
		\begin{equation}\label{eq_1}
		    1- \delta \leq \underline{d}(\{ n \in \N : \dist(x_n, z_n) < \delta\}) \leq \underline{d}(\{ n \in \N : \dist(x_n, z_n) < \vep\}) .
		\end{equation}	
        Since \eqref{eq_1} holds for every $0<\de<\vep$, we conclude that
		\[d(\{ n \in \N : \dist(x_n, z_n) < \vep\})=1. \qedhere\]
    \end{proof}

    Given a topological dynamical system $(X,T)$, the Besicovitch pseudometric on $X^\infty$ induces the \textbf{dynamical Besicovitch pseudometric} on $X$ which is defined for $x,z \in X$ as
    $$\distBT(x,z) = \dist_B(\orbx,\orbz).$$
    
    We define the notions of Cauchy sequence and completeness with respect to the dynamical Besicovitch pseudometric on $X$ in the usual way.
	\begin{defi} Let $(X,T)$ be a topological dynamical system.
		We say that a sequence $\{z^{(n)}\}_{n=1}^\infty \subset X$ is \textbf{Besicovitch-Cauchy} if for every $\vep>0$ there exists $N \in \N$ such that for every $n,m \geq N$ we have
		$$\distBT(z^{(n)},z^{(m)}) <\vep.$$
	\end{defi}
 
	\begin{defi}
		We say that a topological dynamical system $(X,T)$ is \textbf{Besicovitch complete} if for every Besicovitch-Cauchy sequence $\{z^{(n)}\}_{n=1}^{\infty} \subset X$ there exists $z \in X$ (not necessarily unique) such that
		$$\lim_{n \to \infty }\distBT(z^{(n)},z) = 0.$$
	\end{defi}

    \begin{remark}
        The notion of Cauchy sequences and completeness for the Besicovitch pseudometric on $X^\infty$ could also be defined accordingly, and replicating the proof of Proposition 2 in \cite{BFK} we see that $X^\infty$ is always Besicovitch complete.
    \end{remark}
	
    \section{Pseudo-orbits in topological dynamical systems}\label{section_pseudo_orbits}
    Usually, pseudo-orbits are defined together with their corresponding shadowing-like properties. In our case, we will analyze some pseudo-orbits on their own, apart from their corresponding tracing properties. To avoid unnecessary definitions, we decided to devote this section to generalizations of pseudo-orbits and their relations.

    Originally, pseudo-orbits were defined as sequences of points such that, starting from the second point in the sequence, every point is close to the image of the previous one. Roughly speaking, a pseudo-orbit is distinguished from an actual orbit only by small errors at each iteration. 

    \begin{defi}\label{def_pseudoorbit}
    A sequence $\seqx\in X^\infty$ is called a \textbf{$\de$-pseudo-orbit} in $(X,T)$ if $\dist(T(x_n),x_{n+1})<\de$ for every $n \in \N_0.$
    \end{defi}
    
    Allowing the errors to be small in a different sense than the uniform one, such as vanishing in the limit or on average, lead to definitions of different types of pseudo-orbits (we refer to \cite{Blank,limit_shadowing, Pilyugin, Gu} for more details). One of these generalized notions is the asymptotic pseudo-orbit.

    \begin{defi}\label{aysmptotic-p.o}
        A sequence $\seqx\in X^\infty$ is called an \textbf{asymptotic pseudo-orbit} in $(X,T)$ if 
        \[
            \lim\limits_{n\to\infty}\dist(T(x_n),x_{n+1})=0.
        \]
    \end{defi}
    
    Blank \cite{Blank} analyzed pseudo-orbits, nowadays called average pseudo-orbits, in which the average of the errors becomes small in the long term. 
    
    \begin{defi}\label{delta-APO}
		A sequence $\seqx \in X^\infty$ is said to be a \textbf{$\de$-average pseudo-orbit} in $(X,T)$ if there exists $N \in \N$ such that for every $n \geq N$ and $k \in \N_0$ it holds that $$ \frac{1}{n}\sum_{i=0}^{n-1}\dist(T(x_{i+k}), x_{i+k+1}) < \de.$$ 
	\end{defi}

    From Definitions \ref{aysmptotic-p.o} and \ref{delta-APO} it is not hard to see that asymptotic pseudo-orbits are $\de$-average pseudo-orbits for any $\de>0$.
    
    \begin{prop}\label{limiting-p.o-vs-aver}
        If $\seqx\in X^\infty$ is an asymptotic pseudo-orbit in $(X,T)$, then $\seqx$ is a $\delta$-average pseudo-orbit in $(X,T)$ for every $\delta>0$.
    \end{prop}
        
    Later, Gu \cite{Gu} studied the case where the average of errors vanishes in the limit.

    \begin{defi}\label{AAPO}
		A sequence $\seqx \in X^\infty$ is said to be an \textbf{asymptotic average pseudo-orbit} in $(X,T)$ if$$\limsup_{n \to \infty} \frac{1}{n}\sum_{i=0}^{n-1}\dist(T(x_i), x_{i+1}) = 0.$$ 
	\end{defi}

    Combining Remark \ref{equivalence_pseudometrics} and Proposition \ref{lem:eq-of-besicovitch-and-density} we obtain the following result.
    
    \begin{prop}\label{prop:equiv-of-aapo}
        A sequence $\seqx \in X^\infty$ is an asymptotic average pseudo-orbit in $(X,T)$ if and only if for every $\vep >0$ the following holds $$d(\{n \in \N_0 : \dist(T(x_n), x_{n+1})< \vep\})= 1.$$ 
    \end{prop}

     Asymptotic pseudo-orbits and asymptotic average pseudo-orbits are different.
    However, if a topological dynamical system is chain mixing, then given an asymptotic average pseudo-orbit one can obtain an asymptotic pseudo-orbit by making small modifications.

    \begin{defi}
        We say that a topological dynamical system $(X,T)$ is \textbf{chain mixing} if for any $\de>0$ and $x,y \in X$ there is $N \in \N$ such that for every $n \in \N$ satisfying $n \geq N$ there exists a sequence $\{x_i\}_{i=0}^n \subset X$ with $x_0= x$ and $x_n = y$ so that $\dist(T(x_i),x_{i+1})<\de$ for all $i \in [0,n)$.
    \end{defi}
    
    \begin{prop}[Lemma 3.3 in \cite{KKO}]\label{AAPO_asympPO}
        Let $(X,T)$ be a chain mixing topological dynamical system. 
        For every asymptotic average pseudo-orbit $\seqx \in X^\infty$ there exists an asymptotic pseudo-orbit $\seqz \in X^\infty$ such that
        $$d\left( \left\{n\in \N_0 : x_n \neq z_n\right\} \right) = 0.$$
        In particular, we have $\dist_B(\seqx,\seqz) = 0$.
    \end{prop}

    In order to define the vague specification property, Kamae \cite{Kamae} studied an unnamed family of sequences in $X^\infty$ which we will call vague pseudo-orbit.
    
    \begin{defi}\label{VPO}
		We call $\seqx \in X^\infty$ a \textbf{vague pseudo-orbit} in $(X,T)$ if for any open neighborhood $\SU$ of $\Ga_T$ in $X^\infty$ we have
		$$d(\{ n \in \N_0: S^n(\seqx)\in \SU\})=1.$$
	\end{defi}

    It turns out that asymptotic average pseudo-orbits and vague pseudo-orbits are the same sequences. In order to prove the equality between asymptotic average pseudo-orbits and vague pseudo-orbits we will need some lemmas. 

    The first lemma is a consequence of the thickness of subsets of $\N_0$ with asymptotic density $1$. 

    \begin{lemma}\label{intervals}
		Let $G \subset \N_0$. If $d(G)=1$, then for every $k \in \N$ the following holds
		$$d( \{n \in \N_0 : [n,n+k) \subset G \}) = 1.$$
	\end{lemma}
    \begin{proof}
        Let $G \subset \N_0$ be such that $d(G)=1$ and $k \in \N$. Set $G_k= \{n \in \N_0 : [n,n+k) \subset G\}$. It is enough to prove that $\dbar(\N_0 \setminus G_k) = 0$. 
        
        Note that $G_k = \bigcap_{i=0}^{k-1} (G-i).$ 
        Note also that $d(G) = d(G-i)$ for every $i \in [0,k)$. As a consequence, we have $\dbar(\N_0\setminus (G-i))=0$ for every $i \in [0,k)$.
        Lastly, since $\dbar(A\cup B)\leq \dbar(A) + \dbar(B)$ for every $A,B \subseteq \N_0$, we complete the proof by observing that 
        $$\N_0 \setminus G_ k = \bigcup_{i=0}^{k-1}\Big(\N_0 \setminus \left(G-i\right)\Big),$$
        which implies, $\dbar(\N_0 \setminus G_ k)=0$.
    \end{proof}
    
    Directly from the definition of $\dist^\infty$ we have the following lemma.

    \begin{lemma}\label{k_large}
        Let $\vep>0$. If $\frac{1}{2^{k}} < \vep$ and $\seqz \in X^\infty$ satisfies $\max_{1 \leq n \leq k} \dist(z_n, T^n(z_0))< \vep$, then $\dist^\infty(\seqz, \underline{z_0}_T) < \vep$.
    \end{lemma}

    The last ingredient follows from the compactness of the space of orbits in the product topology. 
    Informally, it says that for any open set $\SU \subset X^\infty$ containing $\Ga_T$, there exists a sufficiently large $k \in \N$ such that if the first $k$ terms of $\seqz \in X^\infty$ are sufficiently close to an orbit, then $\seqz$ belongs to $\SU$.

	\begin{lemma}\label{open_nbhd}
		Let $\SU \subset X^\infty$ be an open set containing $\Ga_T$. Then there exist $\vep >0$ and $k \in \N$ such that 
		$$ \left\{ \seqz \in X^\infty : \max_{1 \leq n \leq k} \dist(z_n, T^n(z_0))< \vep \right\} \subset \SU.$$  
	\end{lemma}

    \begin{proof}
    For $\vep>0$ and $x \in X$, we denote
    $$B^\infty(x,\vep) =\{\underline{y} \in X^\infty : \dist^\infty(\underline{y}, \orbx)< \vep\}.$$
    
    Fix $\SU \subset X^\infty$ an open set containing $\Ga_T$. For each $ x \in X$ there exists $\vep = \vep(x)>0$ such that $B^\infty(x,2\vep) \subset \SU$. By compactness of $\Ga_T$, there exist $x^{(1)},\ldots, x^{(m)}\in X$ and $\{\vep_1, \ldots, \vep_m\}$ such that 
	$$\Ga_T \subset \bigcup_{i=1}^m B^\infty(x^{(i)},  \vep_i) \subset \bigcup_{i=1}^m B^\infty(x^{(i)},  2\vep_i) \subset \SU.$$
	
	Let $\vep= \min\{\vep_1, \ldots, \vep_m\}$ and $k \in \N$ be such that $\frac{1}{2^k} < \vep$. Take $\seqz \in X^\infty$ such that 
	$$\max_{1 \leq n \leq k} \dist(z_n, T^n(z_0))< \vep.$$
	
	By Lemma \ref{k_large}, we have $\dist^\infty(\seqz, \underline{z_0}_T) < \vep.$
    Note that $\underline{z_0}_T \in B^\infty(x^{(i)},\vep_i)$ for some $i \in \{1,\ldots,n\}$. Therefore,  by the triangle inequality, $\seqz \in B^\infty(x^{(i)}, 2\vep_i) \subset \SU $.
    \end{proof}

    \begin{theorem}\label{thm:on-equi-of-v.p.o-a.a.p.o}
        Let $(X,T)$ be a topological dynamical system. A sequence $\seqx \in X^\infty$ is an asymptotic average pseudo-orbit in $(X,T)$ if and only if $\seqx$ is a vague pseudo-orbit in $(X,T)$.
    \end{theorem}

	\begin{proof}
        Let $\vep >0$ and $\seqx \in X^\infty$ be a vague pseudo-orbit. By Proposition \ref{prop:equiv-of-aapo}, it suffices to show that 
        \begin{equation}\label{eqn:densitiy-defn-aapo}
            d(\{n \in \N_0 : \dist(T(x_n), x_{n+1}) <\vep\}) =1.
        \end{equation}

        Using the uniform continuity of $T$, take $0<\de<\frac{\vep}{2}$ such that $\dist(T(w),T(v)) <~\frac{\vep}{2}$ for every $w,v \in X$ with $\dist(w,v)<\de$.
        
        Set $\SU=\{\underline{y} \in X^\infty : \exists z \in X \text{ s.t. } \dist(y_0, z)<\de \text{ and } \dist(y_1, T(z))<\de\}$. Note that $\SU$ is an open set containing $\Ga_T$. Since $\seqx$ is a vague pseudo-orbit, by definition, we have
        $$d(\{n \in \N_0 : S^n(\seqx) \in \SU\}) =1.$$

        Observe that for each $n \in \N_0$ such that $S^n(\seqx) \in \SU$ there exists $z_n \in X$ such that $\dist(x_n, z_n)<\de$ and $\dist(x_{n+1}, T(z_n)) < \de $. Thus, by the triangle inequality, we obtain
        $$\dist(T(x_n), x_{n+1}) \leq \dist(T(x_n), T(z_n))+ \dist(x_{n+1}, T(z_n)) < \vep.$$

        Therefore,
        $$\{n \in \N_0 : S^n(\seqx) \in \SU\} \subset \{n \in \N_0 : \dist(T(x_n), x_{n+1}) <\vep\},$$
        and consequently, equation \eqref{eqn:densitiy-defn-aapo} holds.
        
		For the converse, let $\seqx \in X^\infty$ be an asymptotic average pseudo-orbit in $(X,T)$. Fix $\vep >0$ and $k \in \N$. Set
		$$\SV = \left\{ \seqz \in X^\infty : \max_{1 \leq j \leq k} \dist(z_j, T^j(z_0))< \vep \right\}. $$

        By Lemma \ref{open_nbhd}, to prove that $\seqx$ is a vague pseudo-orbit, it suffices to show that
		\begin{equation}\label{eqn:cond-for-vpo}
		  d(\{n \in \N_0 : S^n(\seqx) \in \SV\}) = 1.  
		\end{equation}

        Using uniform continuity of $T$, take $0< \de < \frac{\vep}{k}$ so that for any $w,v \in X$ with $\rho(w,v)<\delta$, we have $\dist(T^j(w), T^j(v)) <\frac{\vep}{k}$ for $j \in [1,k]$.

        Set $$\tilde{G} = \{n \in \N_0 : \dist(T(x_n), x_{n+1}) < \de\}.$$ 
        
        Since $\seqx$ is an asymptotic average pseudo-orbit, it follows from Proposition \ref{prop:equiv-of-aapo} that $d(\tilde{G}) =1$. 
        Note that, by the choice of $\de$, for every $m \in \tilde{G}$ and $j \in [1,k]$ we have
        \begin{equation}\label{eq_unif_cont}
            \dist\left(T^j(x_m), T^{j-1}(x_{m+1})\right) < \frac{\vep}{k}.
        \end{equation}
        
        Let $G = \{n \in \N_0 : [n,n+k) \subset \tilde{G} \}$.
        By Lemma \ref{intervals}, it follows that $d(G)=1$.     
        Fix $n \in G$. By construction,  $(\ref{eq_unif_cont})$ holds for every $m \in [n, n + k)$ and $j \in [1,k]$. 
        
        By the triangle inequality, for every $j \in [1,k]$ we obtain the following estimate (see Figure \ref{triangle}):
        $$\dist\left(T^j(x_n), x_{n+j}\right) \leq \sum_{i=0}^{j-1} \dist \left(T^{j-i}(x_{n+i}), T^{j-i-1} (x_{n+i+1})\right)< \vep,$$
        that is,
		$$\max_{1 \leq j \leq k} \dist\left(T^j(x_n), x_{n+j}\right) < \vep.$$ 
		
        Therefore, $G \subseteq \{n \in \N_0 : S^n(\seqx) \in \SV\},$
		showing that equation \eqref{eqn:cond-for-vpo} holds.
	\end{proof}
    \begin{figure}[ht]
        \centering
        \includegraphics{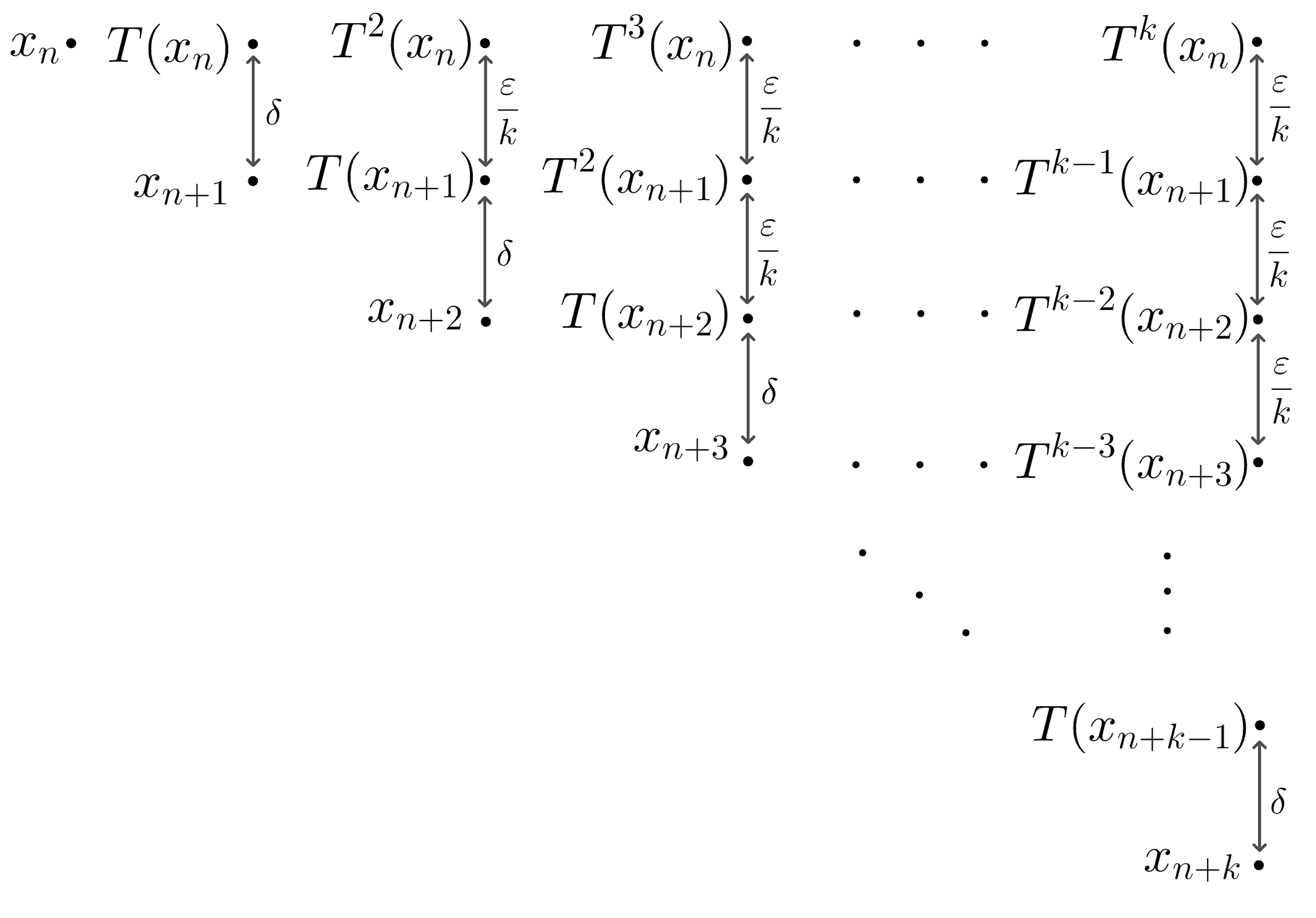}
        \caption{illustration of the triangle inequality.}
        \label{triangle}
    \end{figure}
  
    \section{Tracing properties of topological dynamical systems}\label{section_tracing_properties}

    In the same way as the original definition of pseudo-orbit motivated generalizations, the classical shadowing property inspired alternative types of tracing properties. 
    
    The average shadowing property introduced by Blank \cite{Blank} is one of the alternative notions. 
    \begin{defi}\label{def_asp}
		We say that a topological dynamical system $(X, T)$ has the \textbf{average shadowing property} if for every $\vep >0$ there exists $\de>0$ such that for every $\de$-average pseudo-orbit $\seqx \in X^\infty$ in $(X,T)$ there exists $z \in X$ such that
		$$\dist_B(\seqz_T, \seqx) < \vep.$$
        We say that $\seqx$ is $\vep$-traced on average by $z$. 
	\end{defi} 

    In \cite{KKO}, Kulczycki, Kwietniak and Oprocha provided an equivalent formulation for the average shadowing property for the case where the topological dynamical system is chain mixing.
    It is not hard to see that $\de$-pseudo-orbits are also $\de$-average pseudo-orbits, and roughly speaking, Kulczycki, Kwietniak and Oprocha showed that it is enough to assure that pseudo-orbits can be traced on average.
    
    \begin{theorem}[Theorem 3.6 in \cite{KKO}]\label{thm_ASP_enough_pseudoorbit}
     Let $(X, T)$ be a topological dynamical system. If $(X,T)$ is chain mixing, then the following conditions are equivalent:
    \begin{enumerate}
        \item $(X,T)$ has the average shadowing property;

        \item for every $\vep>0$ there is $\de>0$ such that for every $\de$-pseudo-orbit $\seqx \in X^\infty$ there exists $z \in X$ so that $\dist_B(\seqz_T, \seqx) < \vep.$
    \end{enumerate}
   \end{theorem}
    
    Inspired by the works of Blank \cite{Blank}, and Eirola, Nevanlinna and Pilyugin \cite{limit_shadowing}, Gu introduced the asymptotic average shadowing property in \cite{Gu}.
    
	\begin{defi}\label{def_aasp}
		We say that a topological dynamical system $(X, T)$ has the \textbf{asymptotic average shadowing property} if for every asymptotic average pseudo-orbit $\seqx \in X^\infty$ in $(X,T)$ there exists $z \in X$ such that
		$$\dist_B(\seqz_T, \seqx)= 0.$$
	\end{defi} 

    The average shadowing property is known to be a consequence of the asymptotic average shadowing property (see \cite{KKO, WOC}). One can see that the converse is true provided that $(X,T)$ is surjective and Besicovitch complete.
    \begin{theorem}\label{equiv_AASP_ASP+Besicovitch}
        Let $(X,T)$ be a surjective topological dynamical system. Then $(X,T)$ has the asymptotic average shadowing property if and only if $(X,T)$ is Besicovitch complete and has the average shadowing property.
    \end{theorem}
    \begin{proof}
        By Corollary 4.2 in \cite{WOC}, the asymptotic average shadowing property implies the average shadowing property. Moreover, by Corollary 30 in \cite{LK17}, if $(X,T)$ has the asymptotic average shadowing property, then $(X,T)$ is Besicovitch complete.

        Conversely, assume that $(X,T)$ is Besicovitch complete and has the average shadowing property. Let $\seqx\in X^\infty$ be an asymptotic average pseudo-orbit. By Lemma 3.1 in \cite{KKO} $(X,T)$ is chain mixing, hence, by Proposition \ref{AAPO_asympPO}, there exists an asymptotic pseudo-orbit $\underline{y}\in X^\infty$ such that 
        \begin{equation}\label{eqn:improved-aapo}
            \distB(\seqx,\underline{y})=0.
        \end{equation}
        
        It follows from Proposition \ref{limiting-p.o-vs-aver} that $\underline{y}$ is a $\delta$-average pseudo-orbit in any $\delta>0$. Therefore, by the average shadowing property, for every $\vep>0$ there exists a point $z^{(\vep)} \in X$ such that $\distB(\orbz^{(\vep)},\underline{y})<\vep.$
        
        Using the above observation we pick a sequence $\{z^{(n)}\}_{n=1}^\infty\subset X$ such that $\distB(\underline{z}_T^{(n)},\underline{y})<2^{-n}$ for every $n\in \N$, and for $m>n$ we have $\distBT(z^{(n)},z^{(m)})<2^{-n}.$
        
        The Besicovitch completeness of $(X,T)$ implies that there exists $z \in X$ such that 
        \[
        \lim_{n\to \infty}\distBT(z^{(n)},z)=0.
        \]
        Moreover, by construction, we have
        \[  \lim_{n\to \infty}\distB(\underline{z}_T^{(n)},\underline{y})=0.
        \]
        Therefore, $\distB(\orbz,\underline{y})=0$, and by \eqref{eqn:improved-aapo}, we conclude that $\distB(\orbz,\seqx)=0$.
    \end{proof}
    
    Finally, we focus on the vague specification property, introduced by Kamae in \cite{Kamae}.
    
	\begin{defi}\label{def_vsp}
		We say that a topological dynamical system $(X, T)$ has the \textbf{vague specification property} if for every vague pseudo-orbit $\seqx \in X^\infty$ in $(X,T)$ there is $z \in X$ such that for every $\vep>0$ the following holds:
		$$d\left(\{ n \in \N_0 : \rho(T^n(z), x_n)<\vep \}\right)=1.$$
	\end{defi}

    If we combine Propositions \ref{lem:eq-of-besicovitch-and-density} and \ref{equivalence_pseudometrics}, we see that the notions of tracing used in Definitions \ref{def_aasp} and \ref{def_vsp} are actually equivalent. As a consequence of Theorem \ref{thm:on-equi-of-v.p.o-a.a.p.o}, we obtain the equivalence between Definitions \ref{def_aasp} and \ref{def_vsp}.
	
	\begin{theorem}\label{thm:AASP-VSP}
		A topological dynamical system $(X,T)$ has the vague specification property if and only if $(X,T)$ has the asymptotic average shadowing property.
	\end{theorem}
	
	\begin{proof}
        We first assume that $(X,T)$ has the vague specification property. Fix an asymptotic average pseudo-orbit $\seqx \in X^\infty$. By Theorem \ref{thm:on-equi-of-v.p.o-a.a.p.o}, $\seqx$ is also a vague pseudo-orbit. Hence, by the vague specification property, there exists $z \in X$ such that for every $\vep>0$ it holds that
		$$d\left(\{ n\in \N_0 : \rho(T^n(z), x_n)<\vep \}\right)=1.$$
        
        Applying Lemma \ref{lem:eq-of-besicovitch-and-density} followed by Lemma \ref{equivalence_pseudometrics}, we see that $\dist_B(\orbz ,\seqx)= 0$. Therefore, $(X,T)$ has the asymptotic average shadowing property.

        Conversely, assume that $(X,T)$ has the asymptotic average shadowing property and take $\seqx \in X^\infty$ which is a vague pseudo-orbit. 
        By Theorem \ref{thm:on-equi-of-v.p.o-a.a.p.o}, $\seqx \in X^\infty$ is also an asymptotic average pseudo-orbit in $(X,T)$. 
        Since $(X,T)$ has the asymptotic average shadowing property, there exists $z \in X$ such that $\rho_B(\seqz_T,\seqx)=0$. Applying Lemmas \ref{equivalence_pseudometrics} and \ref{lem:eq-of-besicovitch-and-density}, we obtain 
        $$d\left(\{ n\in \N_0 : \rho(T^n(z), x_n)<\vep \}\right)=1.$$ 
        
        Hence, $(X,T)$ has the vague specification property. \qedhere
	\end{proof}

    \section{Tracing properties of shift spaces}\label{section_shift_spaces}

Let $\alf = \{0,1\}$ be endowed with the discrete topology. We call the set $\alf$ an \textbf{alphabet} and each element of $\alf$ we call a \textbf{letter}.
By the \textbf{full shift} over $\alf$ we mean the infinite Cartesian product $\FS$
of $\alf$ indexed by $\N_0$, that is, 
\[
\FS=\left\{ \left\{x_n\right\}_{n=0}^\infty : x_n \in \alf \text{ for all } n\in \N_0  \right\}.
\]

We endow the full shift $\FS$ with the product topology. By the Tychonoff Theorem, the space $\FS$ is compact. A metric compatible with the product topology on $\FS$ is given for $x=\{x_n\}_{n=0}^\infty, y=\{y_n\}_{n=0}^\infty \in \FS$ by
\[
\rho(x,y)=\begin{cases}
            0, & \mbox{if } x=y, \\
            2^{-\min\{n \in \N_0:x_n\neq y_n\}}, & \mbox{otherwise}.
          \end{cases}
\] 

The \textbf{shift map} $\sigma\colon\FS\to\FS$ is defined by $\sigma(x)=\{x_{n+1}\}_{n=0}^\infty$ where $x=\{x_n\}_{n=0}^\infty\in\FS$.
A nonempty, closed, and $\sigma$-invariant set $X \subseteq \FS$ is called a \textbf{shift space} over $\alf$. When referring to a shift space and its properties, we will be implicitly assuming that we are dealing with the topological dynamical system $(X,\si)$. For aesthetic reasons, along this section we will adopt an upper index for sequences of points in $X$, that is, we write
$$X^\infty = \{ \{x^{(n)}\}_{n=0}^\infty : x^{(n)} \in X \text{ for all } n\in \N_0\}.$$

We call a finite sequence of letters from $\alf$ a \textbf{word} over $\alf$. 
The number of letters in a word $w$ is called the \textbf{length} of $w$. 
We write $|w|$ for the length of a word $w$.
The only sequence with length $0$ is called the \textbf{empty word} and is denoted by $\emptyword$. 

Let $x=\{x_n\}_{n=0}^\infty \in \FS$ and $i, j \in \N_0$. For $ i \leq j$, we write $x_{[i,j)}$ for the word $x_ix_{i+1}\ldots x_{j-1}$ over $\alf$ of length $j-i$ and we agree that $ x_{[i,i)} = \la$. We will also write $x_{[k,\infty)}$ to denote $\si^k(x)$ for $k \in \N_0$. 

We write $\alf^*$ for the set of all words over $\alf$. Each word $w \in \alf^*$ determines the \textbf{cylinder set} $[w] = \{x \in \alf^\infty : x_{[0, |w|)} = w\}$. 
The \textbf{language} of a shift space $X$ is $\lang(X)=\left\{ w\in \alf^* : [w] \cap X \neq \emptyset\right\}$.

For every $n\in\N_0$, we denote by $\lang_n(X) \subset \alf^*$ the set of all words $w\in\lang(X)$ such that $|w|=n$. Hence, the following holds
\[
    \lang(X)=\bigcup_{n\in\N_0}\lang_n(X).
\]

By the \textbf{concatenation} of the words $u=u_1\ldots u_k\in \alf^*$ and $v=v_1\ldots v_m\in \alf^*$ we mean the word $u_1\ldots u_k v_1\ldots v_m$, and we denote it by $uv$. We say that a word $u \in \alf^*$ is a \textbf{prefix} (resp. \textbf{suffix}) of a word $w \in \alf^*$ if there exists $v \in \alf^*$ such that $uv = w$ (resp. $vu=w$). For $n\in \N$, the word $0^n = 0\ldots0 \in \alf^*$ (resp. $1^n = 1 \ldots 1$) stands for the word of length $n$ composed exclusively of $0$'s (resp. $1$'s). If $w \in \alf^*$, then $w^\infty = www\ldots \in \alf^\infty$ represents the infinite self-concatenation of $w$.

We say that a shift space $X$ is \textbf{transitive} if for every pair of words $u,v \in \lang(X)$ there exists $w \in \lang(X)$ such that $uwv \in \lang(X)$. We say that a shift space $X$ is \textbf{mixing} if for every pair of words $u,v \in \lang(X)$ there exists $N \in \N$ such that for every $n \geq N$ there is a word $w \in \lang_n(X)$ such that $uwv \in \lang(X)$.

Let $\mathcal{B}\subset \alf^*$ be a collection of words. We say that $\mathcal{B}$ is a \textbf{code} for a shift space $X$, if $X$ is the collection of all possible infinite concatenations of words from $\SB$ and their shifts. Moreover, $X$ is called \textbf{coded shift} in this case.

We denote an oriented $\alf$-labelled (multi) graph by a triple $G=\left(V,E,\tau\right)$, where $V$ is a set of vertices, $E$ is a set of edges, and $\tau\colon E\to \alf$ is a so-called label map. Every $e \in E$ has two end points, an initial vertex $i(e)\in V$ and terminal vertex $t(e)\in V$. Let $\ell \in \N$. A sequence of edges $\{e_1, \ldots, e_\ell\}\subset E$ is called a \textbf{path} of length $\ell$ if there exists a sequence of vertices $\{v_1, \ldots, v_\ell\} \subset V$ with $i(e_i)=v_i$, and $t(e_i)=v_{i+1}$ for every $1 \leq i <\ell$. 

We define the shift $X_G$ to be the set of all sequences $x\in \FS$ such that $x$ is the sequence of labels of an infinite path on a graph $G$. A shift space $X$ is called \textbf{sofic} if there exists a labelled graph $G=\left(V,E,\tau \right)$ such that the set of vertices $V$ is finite and $X=X_G$. 

The $\textbf{hereditary closure}$ of a shift space $X$ is
\begin{equation*}\label{eqn:hereditarty}
    \tilde{X}=\{ y \in \FS : \exists x\in X \text{ s.t. } y\leq x \text{ coordinate-wise}\}.
\end{equation*} A shift space $X$ is called \textbf{hereditary} if $\tilde{X}=X$.

We say that $w \in \lang(X)$ is a \textbf{synchronizing word} for a shift space $X$ if $uw, wv \in \lang(X)$ implies $uwv\in \lang(X)$. A shift space $X$ is called \textbf{synchronized} if it is transitive and has a synchronizing word.
In the particular case of hereditary shifts, it is not hard to see that we can assume the synchronizing word to be composed only of $0$s. More precisely, we have the following lemma whose proof is left to the reader.
\begin{lemma}\label{hereditary_shift_sync_word}
    Let $X$ be a hereditary shift space. If $w \in \lang_n(X)$ is a synchronizing word for $X$, then $0^n \in \lang(X)$ is also a synchronizing word for $X$.
\end{lemma}

As a consequence of Lemma \ref{hereditary_shift_sync_word}, we see that if a surjective shift space $X$ is hereditary and synchronized, then there exists a fixed integer $n \in \N$ such that we can connect any word $u \in \lang(X)$ with any point $x \in X$ using a word $0^n$.

\begin{lemma}\label{hereditary_shift_spec_sync_word}
    Let $X$ be a hereditary surjective shift space and $w \in~\lang_n(X)$ be a synchronizing word for $X$. If $u\in\lang(X)$ and $x \in X$, then $u0^nx \in X$. 
\end{lemma}

 \begin{proof}
     Note that, by hereditarity, $u0^m \in \lang(X)$ for every $u \in \lang(X)$ and $m \in \N$. 
    
     Since $X$ is surjective, for every $x \in X$ and $n \in \N$ there exists a word $v \in \lang_n(X)$ such that $vx \in X$.
     By hereditarity, we can assume that $v=0^n$, that is, $0^nx \in X$. 
    
     It follows from Lemma \ref{hereditary_shift_sync_word} that $0^n$ is a synchronizing word for $X$. Hence, for every $v \in \lang(X)$ that is prefix of $x$ we have $u0^nv \in \lang(X)$, and consequently, $u0^nx \in X$. 
 \end{proof}

The original specification property introduced by Bowen \cite{Bowen} for general topological dynamical systems has an equivalent formulation for shift spaces.

\begin{defi}
    Let $X$ be a shift space. We say that $X$ has the \textbf{specification property} if there exists $n\in \N$ such that for every $u,v\in \lang(X)$ there is $w\in\lang_n(X)$ such that $uwv\in \lang(X)$. 
\end{defi}

The class of shift spaces with the specification property is quite broad and contains, for example, the class of mixing sofic shifts.

\begin{prop}[Section 6 of \cite{Weiss}]
    Every mixing sofic shift space has the specification property.
\end{prop}

\begin{remark}
    We remark that there are mixing shift spaces without the specification property (see \cite{Jung, KLO}). 
\end{remark}

In 1988, Bertrand \cite{Bertrand} proved that shift spaces with specification always have synchronizing words.

\begin{theorem}[Theorem 1 in \cite{Bertrand}]
    Let $X$ be a shift space. If $X$ has the specification property, then $X$ is synchronized.
\end{theorem}

We will abuse the notation and denote by $\dbar$ also the pseudometric on $\FS$ induced by the upper asymptotic density given for $x,y \in \FS$ by
$$\dbar(x,y) = \dbar(\left\{ n \in \N_0: x_n \neq y_n\right\}) = \limsup_{n \to \infty} \frac{1}{n} \left|\left\{i \in [0,n) : x_i \neq y_i\right\}\right|. $$

\begin{prop}[Corollary 5 in \cite{KLO2}]\label{equivalence_dbar_besicovitch}
    The dynamical Besicovitch and the $\dbar$ pseudometrics on $\alf^\infty$ are uniformly equivalent.
\end{prop}

We will discuss convergence of sequences of points in $\FS$ with respect to the $\dbar$ pseudometric. We remark that since $\dbar$ is a pseudometric, a limit point, if it exists, might not be unique. We define the notion of Cauchy sequence and completeness with respect to the $\dbar$ pseudometric in the natural way.
    
	\begin{defi}
		We say that a sequence $\{z^{(n)}\}_{n=1}^\infty \subset X$ is \textbf{$\dbar$-Cauchy} if for every $\vep>0$ there exists $N \in \N$ such that for every $n,m \geq N$ we have
		$$\dbar(z^{(n)},z^{(m)}) <\vep.$$
	\end{defi}
	
	\begin{defi}
		We say that a shift space $X$ is \textbf{$\dbar$-complete} if for every $\dbar$-Cauchy sequence $\{z^{(n)}\}_{n=1}^{\infty} \subset X$ there exists $z \in X$ such that
		$$\lim_{n \to \infty }\dbar(z^{(n)},z) = 0.$$
	\end{defi}

    By the uniform equivalence between the Besicovitch and $\dbar$ pseudometrics (Proposition \ref{equivalence_dbar_besicovitch}), it follows from Corollary 30 in \cite{LK17} that the asymptotic average shadowing property is a sufficient condition for $\dbar$-completeness.
    
    \begin{prop}
        Let $X$ be a shift space. If $X$ has the asymptotic average shadowing property, then $X$ is $\dbar$-complete.
    \end{prop}
    
	In the particular case of shift spaces, a notion of tracing called $\dbar$-shadowing property, which is closely related to the average shadowing property, was introduced by Konieczny, Kupsa and Kwietniak in \cite{KKK}.
    \begin{defi}
        We say that a shift space $X$ has the \textbf{$\dbar$-shadowing property} if for every $\vep >0$ there exists $N \in \N$ such that for every sequence $\{w^{(i)}\}_{i=1}^\infty \subset \SL(X)$ with $|w^{(i)}| \geq N$ for every $i \in \N$ there exists $x \in X$ such that 
        $$\dbar(x,w)< \vep,$$
        where $w=w^{(1)}w^{(2)}\ldots \in \alf^\infty$. 
    \end{defi}

    \begin{remark}
        We remark that given a sequence of words $\{w^{(i)}\}_{i=1}^\infty \subset \SL(X)$ their concatenation $w = w^{(1)} w^{(2)}\ldots$ might not belong to $X$.
    \end{remark}

     As mentioned in \cite{KKK}, a shift space with the average shadowing property also possesses the $\dbar$-shadowing property. Moreover, if we assume surjectivity of the shift space, then the $\dbar$-shadowing property implies the average shadowing property. Due to the lack of proofs for these statements, for the sake of completeness, we will prove them.

    \begin{lemma}\label{lemma_N_dbar_apo}
	For every $\de>0$ there exists $N \in \N$ such that if a sequence of words $\{w^{(i)}\}_{i=1}^\infty \subset \lang(X)$ satisfies $|w^{(i)}| \geq N$ for every $i \in \N$, then there exists a $\de$-average pseudo-orbit $\seqx \in X^\infty$ such that $$\dist_B(\seqw_\si, \seqx) < \de,$$
	where $w=w^{(1)}w^{(2)}\ldots \in \alf^\infty$.
\end{lemma}

\begin{proof}
	Let $\de >0$. Take $m \in \N$ so that $2^{-m} < \frac{\de}{2}$ and $N \in \N$ satisfying $\frac{m}{N} < \frac{\de}{2}$.
	Fix a sequence of words $\{w^{(i)}\}_{i=1}^\infty \subset \lang(X)$ with $|w^{(i)}| \geq N$ for every $i \in \N$. Let $w = w^{(1)} w^{(2)} \ldots \in \alf^\infty$ denote the concatenation of these words. 
	
	Set $n_1=0$ and $n_i = |w^{(1)} \ldots w^{(i-1)}|$ for $i \geq 2$. We now construct the sequence $\seqx=\{ x^{(n)} \}_{n=0}^\infty \in X^\infty$ as follows. For each $n \in \N$ there exists $i= i(n) \in \N$ such that $n \in [n_i, n_{i+1})$. Hence, we pick
	$$x^{(n)} \in \left[w^{(i)}_{[n-n_i, |w^{(i)}|)}\right] \cap X.$$
	
	Set $E = \bigcup_{i \geq 2} (n_i - m, n_i]$. Note that if $n \in \N_0 \backslash E$, then for its correspondent $i$ as above, it holds that
	$$\left|w^{(i)}_{[n-n_i, |w^{(i)}|)}\right| > m.$$ Thus, by construction, for every $n \in \N_0 \backslash E$ we have $\dist(\si^n(w), x^{(n)}) \leq 2^{-m}$. We observe that $\dbar(E) \leq \frac{m}{N}$, which implies $\dist_B(\underline{w}_\si, \seqx) < \frac{\de}{2}$.
	
	It remains to prove that $\seqx \in X^\infty$ is a $\de$-average pseudo-orbit. Note that, also by construction, for every $n \in \N_0\backslash E$ we have $x^{(n)}_{[1,m]} = x^{(n+1)}_{[0,m)}$, that is, 
	\begin{equation}\label{eq_dist_xn_xn+1}
		\dist(\si(x^{(n)}), x^{(n+1)}) \leq 2^{-m}. 
	\end{equation}
 
    Fix $n \geq N$ and $k \in \N_0$. We first find an upper bound for $|[k, k+n) \cap E|$. Observe that in the worst case scenario we have $n_{i+1}-n_i = N$ for all $i \in \N$. Thus,
	\begin{equation}\label{eq_size_En}
		|[k, k + n) \cap E| \leq \frac{nm}{N}< \frac{n\de}{2}. 
	\end{equation}
	
	We can split the sum below into terms in $E$ and outside $E$. Since the diameter of $X$ is assumed to be at most $1$, using \eqref{eq_dist_xn_xn+1} and \eqref{eq_size_En}, we obtain
	\begin{align*}
		\frac{1}{n} \sum_{i=0}^{n-1} \dist(\si(x^{(k+i)}), x^{(k+i+1)}) &\leq \frac{\left(n - |[k, k + n\right) \cap E|)\cdot 2^{-m}}{n}  + \frac{|[k, k + n) \cap E|}{n}\\
		&< 2^{-m} + \frac{\de}{2} < \de. \qedhere
	\end{align*}
\end{proof}

\begin{theorem}\label{thm_ASP_dbar_shadowing}
	Let $X$ be a shift space. If $X$ has the average shadowing property, then $X$ has the $\dbar$-shadowing property.
\end{theorem}

\begin{proof}
	Let $\vep >0$. By the uniform equivalence between the Besicovitch and the $\dbar$ pseudometrics, there exists $0<\vep'<\vep$ such that for every $w,z \in \alf^\infty$ satisfying $\dist_B(\seqw_\si, \seqz_\si) < \vep'$, we have $\dbar(w,z) < \vep$.
	
	Let $0 < \de \leq \frac{\vep'}{2}$ be such that every $\de$-average pseudo-orbit is $\frac{\vep'}{2}$-traced on average by a point in $X$, and let $N \in \N$ be provided by Lemma \ref{lemma_N_dbar_apo}.
	Fix a sequence of words $\{w^{(i)}\}_{i=1}^\infty \subset \lang(X)$ satisfying $|w^{(i)}| \geq N$ for every $i \in \N$. 
	
	By Lemma \ref{lemma_N_dbar_apo}, there exists a $\de$-average pseudo-orbit $\seqx \in X^\infty$ such that $\dist_B(\seqw_\si, \seqx) < \de$. The average shadowing property yields the existence $z \in X$ such that $\dist_B(\seqx, \seqz_\si) < \frac{\vep'}{2}$. Therefore, by the triangle inequality, we obtain $\dist_B(\underline{w}_\si, \seqz_\si) < \vep'$. Finally, by the choice of $\vep'$, we conclude that $\dbar(w,z) < \vep$.
\end{proof}

Our proof of a converse of Theorem \ref{thm_ASP_dbar_shadowing} requires the shift space to be chain mixing. Lemma 9 in \cite{KKK} guarantees that, in our context, surjectivity is sufficient.

\begin{lemma}[Lemma 9 in \cite{KKK}]\label{lemma_chain_mixing}
	Let $X$ be a surjective shift space. If $X$ has the $\dbar$-shadowing property, then $X$ is chain mixing.
\end{lemma}

\begin{theorem}
	Let $X$ be a surjective shift space. If $X$ has the $\dbar$-shadowing property, then $X$ has the average shadowing property.
\end{theorem}

\begin{proof}
	Let $\vep >0$. By the uniform equivalence between the Besicovitch and the $\dbar$ pseudometrics, there exists $\vep'>0$ such that $\dist_B(\seqw_\si, \seqz_\si) < \frac{\vep}{2}$ for any $w,z \in \alf^\infty$ so that $\dbar(w,z) < \vep'$.
	
	Let $N \in \N$ be provided by the $\dbar$-shadowing property for $\vep'$. Without loss of generality, we may assume that $2^{-N} < \frac{\vep}{2}$. By Lemma \ref{lemma_chain_mixing}, $X$ is chain-mixing, so by Theorem \ref{thm_ASP_enough_pseudoorbit}, it remains to prove that there exists $\de >0$ such that every $\de$-pseudo-orbit is $\vep$-traced on average by some point $z \in X$.
	
	Let $m \in \N$ be such that $\frac{N}{m} < \frac{\vep}{2}$ and $0 < \de < \frac{\vep}{2}$ be such that $\de \leq 2^{-m}$. Fix a $\de$-pseudo-orbit $\seqx \in X^\infty$. Note that $\de \leq 2^{-m}$ implies $x^{(n)}_{[1,m]} = x^{(n+1)}_{[0,m)}$ for every $n \in \N_0$. As a consequence, for every $k \in [0,m)$ the following holds
	\begin{equation}\label{eq_PO_consec_terms}
		x^{(n)}_{[k,m)} = x^{(n+k)}_{[0,m-k)}.
	\end{equation}

    For each $i \in \N$ we set $w^{(i)} = x^{(i-1)m}_{[0,m)} \in \lang_m(X)$. Let $w = w^{(1)}w^{(2)}\ldots \in \alf^\infty$ denote the concatenation of $w^{(i)}$'s.

    For every $n \in \N_0$ there exists $i=i(n) \in \N$ such that $n \in [(i-1)m, im)$. It follows from \eqref{eq_PO_consec_terms} that
	\begin{equation}\label{eq_wnim_xn}
		w_{[n, im)} = x^{(i-1)m}_{[n-(i-1)m, m)}= x^{(n)}_{[0, im - n)}.
	\end{equation}
	
	We will use an argument similar to the one in Lemma \ref{lemma_N_dbar_apo} to prove that $\dist_B(\seqw_\si, \seqx) < \de$. Set $E = \cup_{i \in \N} (im - N, im]$. It follows from \eqref{eq_wnim_xn} that for every $n \in \N_0 \backslash E$ it holds that $\dist(\si^n(w), x^{(n)}) \leq 2^{-N}$. Note that $\dbar(E) \leq \frac{N}{m}$, thus $\dist_B(\seqw_\si, \seqx) < \frac{\vep}{2}.$
	
	Now we use the $\dbar$-shadowing property to find $z \in X$ such that 
	$\dbar(w,z) <\vep'$. Lastly, by the triangle inequality, the choice of $\vep'$ implies $\dist_B(\seqx, \seqz_\si) < \vep.$
\end{proof}
            
Since the $\dbar$-shadowing and the average shadowing properties are equivalent for surjective shift spaces, the equivalence of $\dbar$ and Besicovitch pseudometrics, together with Theorem \ref{equiv_AASP_ASP+Besicovitch}, yields the following corollary. 
        
    \begin{corollary}\label{cor:AASP_equiv_dbar-shadow}
        Let $X$ be a surjective shift space. Then $X$ has the asymptotic average shadowing property if and only if $X$ is $\dbar$-complete and has the $\dbar$-shadowing property.  
    \end{corollary}

    \section{Examples of shift spaces with the vague specification property}\label{section_examples}
    \subsection{Proximal shift space with the vague specification property}

    We will show that the shift space $Z$, defined in Example 21 in \cite{CKKK}, has the vague specification property. Since it is known that $Z$ has the $\dbar$-shadowing property, it is enough to show that $Z$ is 
    \hbox{$\dbar$-complete} and use Corollary \ref{cor:AASP_equiv_dbar-shadow}. 
    Before showing that $Z$ is \nolinebreak $\dbar$-complete we recall its construction following \cite{CKKK}.
\begin{example}\label{ex:proximal}
Let $n\in\N$. We construct a sofic shift $Z_n$ represented by an oriented labeled graph $G_n = (V_n,E_n,\tau_n)$ with vertex set $V_n=\{v_0, v_1, \ldots,v_{10^n-1}\}$ whose edge set $E_n$ and labels $\tau_n : E_n \to \{0,1\}$ are given by
\begin{itemize}
\item For every $0 \leq k < 10^n$ there is an edge from $v_k$ to $v_{k+1 \mod 10^n}$ with label $0$.
\item For every $1\leq k\leq 10^n-2^n$ there is an edge from $v_k$ to $v_{k+1}$ with label $1$.
\item There is an edge from $v_{10^n-2^n}$ to $v_{10^n-2^n+2}$ with label $0$.
\end{itemize}
Let $Z = \bigcap_{n=1}^\infty Z_n$. We remark that for every $n \in \N$ we have $Z_{n+1} \not \subset Z_{n}$ and $Z_n \not \subset Z_{n+1}$.
For each $n \in \N$ we set $Y_n=\bigcap_{j=1}^n Z_n$. This defines a decreasing sequence of mixing sofic shifts $Y_1\supseteq Y_2\supseteq \dots$ with $Z=\bigcap_{n=1}^\infty Y_n$. 

We denote
\begin{equation}\label{eq_sets_Ej}
    \tilde{E}_j = \{ n \in \N : n \geq 10^j - 2^j  \mod 10^j \} \quad \text{and} \quad
        E_i = \bigcup_{j \geq i} \tilde{E}_j.
\end{equation}
Note that 
        \begin{equation}\label{ineq_density_En}
            \dbar(E_i) \leq \frac{1}{4 \cdot 5^{i-1}}<5^{-i+1}.
        \end{equation}
        
The crucial observation about $Z$, that we will repeatedly use below, is that if $n \in \N$ and $y \in Y_n$, then it is enough to set $y_j$ to $0$ for every $j \in E_n$ to obtain a point in $Z$.
\end{example}
    
    \begin{theorem}\label{thm_proximal_complete}
        The shift space $Z$ is $\dbar$-complete.
    \end{theorem}

    \begin{proof}
        We let $\tilde{E}_j$ and $E_i$ to be defined as in \eqref{eq_sets_Ej}. Let $\{x^{(n)}\}_{n=1}^\infty$ be a $\dbar$-Cauchy sequence in $Z$. Passing to a subsequence, we can assume that for every $n \in \N$ the following holds
        \begin{equation}\label{ineq_seq_Cauchy_proximal}
            \dbar(x^{(n)},x^{(n+1)}) < 5^{-n-1}.
        \end{equation} 

        By the triangle inequality and \eqref{ineq_seq_Cauchy_proximal}, for every $m \in \N$ with $m \geq n$, we have
        \begin{equation}\label{ineq_Cauchy_xn_xm}
            \dbar(x^{(n)},x^{(m)}) < 5^{-n}.
        \end{equation}
      
        We will construct a strictly increasing sequence of integers $\{\ell_n\}_{n=2}^\infty \subset \N$ and a sequence of points $\{z^{(n)}\}_{n=1}^\infty \subset Z$ so that for every $n \in \N$ the word $z^{(n)}_{[0,\ell_{n+1})}$ is a prefix of $z^{(n+1)}$, that is,
        \begin{equation}\label{aux_seq}
            z^{(n)}_{[0,\ell_{n+1})} = z^{(n+1)}_{[0,\ell_{n+1})}.
        \end{equation}
        
        Using \eqref{aux_seq} we define $x \in Z$ to be the unique point satisfying $x_{[0,\ell_n)} = z^{(n)}_{[0,\ell_n)}$ for every $n\geq 1$. We are going to show that $x$ is a $\dbar$-limit of $\{x^{(n)}\}_{n=1}^\infty$.
        Equivalently, $x$ is given by $$x=z^{(1)}_{[0,\ell_2)}z^{(2)}_{[\ell_2,\ell_3)}\cdots.
        $$

        Note that $Y_n$ is a hereditary shift with specification, in particular, synchronized, for every $n\in \N$. Therefore, by Lemmas \ref{hereditary_shift_sync_word} and \ref{hereditary_shift_spec_sync_word}, for every $n \in \N$ there exists $m_n \in \N$ such that $u_n = 0^{m_n}$ is a synchronizing word for $Y_n$.
        
        We construct an auxiliary sequence of points $\{y^{(n)}\}_{n=1}^\infty$ as
        \begin{equation*}\label{eqn:natural-seq-in-Z}
            y^{(n)}_i=\begin{cases}
                0, \quad &\text{ if } i \in E_{n}, \\
                x^{(n)}_i, &\text{ otherwise.}
            \end{cases}
        \end{equation*}

        Since $Z$ is hereditary, it follows that $y^{(n)} \in Z$ for every $n \in \N$. Observe that, using \eqref{ineq_density_En}, for every $n \in \N$ we have
        \begin{equation}\label{ineq_yn_x_n}
            \dbar(y^{(n)}, x^{(n)}) \leq \dbar(E_n) < 5^{-n+1}.
        \end{equation}
        
        Let $\ell_2 \in \N$ be such that 
        \begin{enumerate}
            \item $\frac{|u_2|}{\ell_2} < 5^{-2}$;

            \item $\left| \left\{ i \in [0,\ell) : y^{(2)}_i \neq x^{(2)}_i \right\}\right| < \ell \cdot 5^{-1}$  for every $\ell \geq \ell_2$ (possible by \eqref{ineq_yn_x_n}).
        \end{enumerate}

        Set $z^{(1)} = y^{(1)}$ and 
        $$z^{(2)} =  y^{(1)}_{[0,\ell_2)}u_{2}y^{(2)}_{[\ell_2 + |u_{2}|, \infty)} \in Y_{2} = Z_1 \cap Z_2.$$
        
        We have $\dbar(z^{(2)}, y^{(2)}) = 0$. To see that  $z^{(2)} \in Z$, it is enough to note that $z^{(2)}_i=0$ for every $i \in E_2$ and $z^{(2)} 
        \in Y_2$. 
        
        We proceed with the construction of $z^{(n)} \in Z$ inductively as follows. For $n > 2$ we will take $\ell_n \in \N$ satisfying the following conditions:
        \begin{enumerate}[label=(\arabic*)]
            \item\label{item:cond-1} $\ell_n \geq 2 \ell_{n-1}$;
            
            \item\label{item:cond-2} $\frac{|u_n|}{\ell_n} < 5^{-n}$;

            \item\label{item:cond-4} $\left| \left\{ i \in [0,\ell) : y^{(n)}_i \neq x^{(n)}_i \right\}\right| < \ell \cdot 5^{-n+1}$  for every $\ell \geq \ell_n$ (possible by \eqref{ineq_yn_x_n});

            \item\label{item:cond-5} $\left| \left\{ i \in [0,\ell) : x^{(n)}_i \neq x^{(m)}_i \right\}\right| < \ell \cdot 5^{-n}$  for every $\ell \geq \ell_n$ (possible by \eqref{ineq_Cauchy_xn_xm});

            \item\label{item:cond-6} $|E_n \cap [0,\ell)| < \ell \cdot 5^{-n+1}$ for every $\ell \geq \ell_n$ (possible by \eqref{ineq_density_En}).
        \end{enumerate}

        We define $z^{(n)}$ by
        \begin{equation*}
           z^{(n)} = z^{(n-1)}_{[0, \ell_n)} u_n y^{(n)}_{[\ell_n + |u_n|, \infty)} \in Y_n = \bigcap_{j=1}^n Z_j.
        \end{equation*}

       For every $n \in \N$ we have $\dbar(z^{(n)},y^{(n)}) = 0$. To see that $z^{(n)} \in Z$, it is enough to note that $z^{(n)}\in Y_n$ and $z^{(n)}_i=0$ for every $i \in E_n$. Also, for every $n \in \N$ we have
        \begin{equation}\label{zn_yn}
            \left| \left\{ i \in [\ell_n,\ell_{n+1}) : z^{(n)}_i \neq y^{(n)}_i \right\}\right| \leq |u_n|.
        \end{equation}
        
        By the triangle inequality, it follows from condition \ref{item:cond-4} that our choice of $\ell_n$, together with \eqref{ineq_yn_x_n} and \eqref{zn_yn}, guarantees
        \begin{align}\label{dist_x_xn_ln}
        \begin{split}
            \left| \left\{ i \in [\ell_n,\ell_{n+1}) : x_i \neq x^{(n)}_i \right\}\right| &= \left| \left\{ i \in [\ell_n,\ell_{n+1}) : z^{(n)}_i \neq x^{(n)}_i \right\}\right| \\
            &\leq \left| \left\{ i \in [\ell_n,\ell_{n+1}) : y^{(n)}_i \neq x^{(n)}_i \right\}\right| + |u_n| \\
            &\leq \ell_{n+1} \cdot 5^{-n+1}+ |u_n|.
        \end{split}
        \end{align}

        Conditions \ref{item:cond-1} and \ref{item:cond-2} mean that our choice of $\ell_n$ guarantees that for every $n, k \in \N$ such that $k>n$ we have\begin{equation}\label{upperbound_sync_words}
            \sum_{j=n}^{k-1} |u_j| \leq \sum_{j=n}^{k-1} \ell_j \cdot 5^{-j} < 2 \ell_{k-1} \cdot 5^{-n} \leq \ell_{k} \cdot 5^{-n}.
        \end{equation}
        
        It remains to show that $\lim_{n \to \infty} \dbar(x^{(n)},x) = 0$.
        Fix $n \in \N$. For every $\ell \geq \ell_{n+1}$ there exists $k \geq n +1 $ such that $\ell \in [\ell_k,\ell_{k+1})$. Note that
        $$\left|\left\{i \in [0,\ell) : x_i \neq x^{(n)}_i \right\}\right| \leq  \left|\left\{i \in [0,\ell_{n+1}) : x_i \neq x^{(n)}_i\right\}\right| + \left|\left\{i \in [\ell_{n+1},\ell) : x_i\neq x^{(n)}_i\right\}\right|.$$
        By the triangle inequality, we have 
        \begin{align*}
            \left|\left\{i \in [\ell_{n+1},\ell) : x_i\neq x^{(n)}_i\right\}\right| \leq &
             \sum_{j=n+1}^{k-1} \left|\left\{i \in [\ell_j,\ell_{j+1}) : x^{(n)}_i\neq x^{(j)}_i\right\}\right| 
             \\& +\sum_{j=n+1}^{k-1} \left|\left\{i \in [\ell_j,\ell_{j+1}) : x_i\neq x^{(j)}_i\right\}\right|  \\ &+
              \left|\left\{i \in [\ell_k,\ell) : x^{(n)}_i \neq x^{(k)}_i\right\}\right|  +
             \left|\left\{i \in [\ell_k,\ell) : x_i \neq x^{(k)}_i\right\}\right|     .
        \end{align*}
        
        By condition \ref{item:cond-5} for the choice of $\ell_n$'s and inequality \eqref{dist_x_xn_ln}, we obtain
        \begin{align}\label{final_sum}
        \begin{split}
            \left|\left\{i \in [0,\ell) : x_i \neq x^{(n)}_i \right\}\right| \leq &\ell_{n+1} 
            +
             \sum_{j=n+1}^{k-1} \ell_{j+1} \cdot 5^{-n}  
            + \sum_{j=n+1}^{k-1} \left(\ell_{j+1} \cdot 5^{-j+1} + |u_j| \right) \\
            & + \ell \cdot 5^{-n}
            + |u_k| + |E_k \cap [\ell_k, \ell)|.
        \end{split}
        \end{align}

        Note that condition \ref{item:cond-1} implies 
        $$\sum_{j=n+1}^{k-1} \ell_{j+1} \cdot 5^{-j+1} \leq \sum_{j=n+1}^{k-1}  \ell_{j+1} \cdot  5^{-n} \leq  2\ell_k \cdot 5^{-n}.$$
        As a consequence, from \eqref{upperbound_sync_words} we obtain 
        \begin{align}\label{eq_sum_x_xj}
         \begin{split}   
            \sum_{j=n+1}^{k-1} \left(\ell_{j+1} \cdot 5^{-j+1} + |u_j| \right) &\leq \sum_{j=n+1}^{k-1} |u_j| + \sum_{j=n+1}^{k-1} \ell_{j+1} \cdot 5^{-j+1} \\ 
            &\leq \ell_{k} \cdot 5^{-n-1} + 2 \ell_{k} \cdot 5^{-n} \\
            &\leq 3\ell_k \cdot 5^{-n}.
         \end{split}
        \end{align}
        
        Therefore, using conditions \ref{item:cond-1} and \ref{item:cond-5} together with \eqref{eq_sum_x_xj}, it follows from \eqref{final_sum} that
        \begin{align}\label{eq_final_sum}
        \begin{split}
           \left|\left\{i \in [0,\ell) : x_i \neq x^{(n)}_i \right\}\right| &\leq \ell_{n+1} + 2\ell_k \cdot 5^{-n} + 3\ell_k \cdot 5^{-n} + |u_k| + |E_k \cap [\ell_k, \ell)|  \\
           & \leq \ell_{n+1} + \ell_k 5^{-n+1}+ |u_k| + |E_k \cap [\ell_k, \ell)|.
        \end{split}
        \end{align}

        Lastly, by the conditions \ref{item:cond-2} and \ref{item:cond-6} for the choice of $\ell_k$, it follows from \eqref{eq_final_sum} that
        \begin{align}
        \begin{split}
            \frac{\left|\left\{i \in [0,\ell) : x_i \neq x^{(n)}_i \right\}\right|}{\ell} &\leq \frac{\ell_{n+1}}{\ell} 
            + \frac{\ell_k \cdot 5^{-n+1}}{\ell}
            + \frac{|u_k|}{\ell} + \frac{|E_k \cap [\ell_k, \ell)|}{\ell} \\ 
            & \leq \frac{\ell_{n+1}}{\ell} + 5^{-n+1} + 5^{-k} + 5^{-k+1}.
        \end{split}
        \end{align}

        Hence,
        $\dbar(x, x^{(n)}) \leq 5^{-n+1}$ for all $n \in \N$, and
        $\lim_{n \to \infty} \dbar(x, x^{(n)}) = 0$.
    \end{proof}

    \begin{corollary}
        The proximal shift space $Z$ has the vague specification property.
    \end{corollary}

    \subsection{A class of minimal shift spaces with the vague specification property}\label{minimal_example}
    In \cite{CKKK}, the authors provided a way of constructing minimal shift spaces that are mixing and have the $\dbar$-shadowing property, acknowledging Piotr Oprocha for the ideas that inspired this construction.
    It is worth mentioning that the obtained shift spaces have positive topological entropy (for definition and details, we refer to \cite{CKKK}).
    We will show that those shift spaces are $\dbar$-complete. 
    
    The parameters of the construction are an initial finite non-empty set of words $\SB_1 \subset \alf^*$ and a sequence $\{t(n)\}_{n=1}^\infty \subset \N$ satisfying $t(n) \geq 2$ for all $n \in \N$.
    
    Assume that we have defined the family of words $\SB_n$ for some $n \in \N$. Write $k(n)$ for the cardinality of $\SB_n$. Enumerate the elements of $\SB_n$ as $\be^{(n)}_1,\ldots,\be^{(n)}_{k(n)}$, and let $\tau(n)=\be^{(n)}_1 \ldots \be^{(n)}_{k(n)}$ denote their concatenation. Let $s(n)$ (resp. $l(n)$) be the length of the shortest (resp. longest) word in $\SB_n$.
    Words belonging to $\SB_{n+1}$ are constructed as follows: first we concatenate $t(n)$ words arbitrarily chosen from $\SB_n$, then we add the suffix $\tau(n)$, that is,
    $$
    \SB_{n+1}=\left\{b_1 b_2 \dots b_{t(n)} \tau(n) : b_i \in\SB_n \text{ for all } i \in [1,t(n)]\right\}.
    $$

   By construction, we have
    \begin{align}\label{smallest_biggest_lenght}
    \begin{split}
        s(n+1) &= t(n)s(n) + |\tau(n)|; \\
        l(n+1) &= t(n)l(n) + |\tau(n)| .
    \end{split}
    \end{align}
    
    For $n \in \N$, let $X_n$ be the coded shift space generated by the code $\SB_n$. That is, $X_n$ consists of all infinite concatenations of words from $\SB_n$ together with their shifts. The proofs of the facts below can be found in Section 4.2 of \cite{CKKK}.

    \begin{fact}\label{fact1}
        Every word in $\SB_n$ is a subword of every word in $\SB_m$, for every $n < m$.
    \end{fact}
    \begin{fact}
        The shift space $X_n$ is transitive and sofic for every $n \in \N$;
    \end{fact}
    \begin{fact}
        For every $n\in \N$ one has $X_{n+1} \subseteq X_n$, which implies $\bigcap_{n=1}^\infty X_n \neq \emptyset$.
    \end{fact}
    \begin{fact}
         The shift space $X = \bigcap_{n=1}^\infty X_n $ is minimal.
    \end{fact}

    From now on, we set $\SB_1=\{0,11\}$. As noted in \cite{CKKK}, from this choice of $\SB_1$, a simple inductive argument shows that, for every $n \in \N$, the set of lengths of all words in $\SB_n$ is an interval.
    \begin{prop}\label{interval_length_words}
        For every $m \in [s(n),l(n)]$, there is $u \in \SB_n$ such that $|u|=m.$
    \end{prop}

    Furthermore, we observe that the ratio between $s(n)$ and $l(n)$ is increasing. In particular, for all $n \in \N$ we have
    \begin{align}\label{sn-ln-ratio}
    \begin{split}
        \dfrac{s(n)}{l(n)}\geq \dfrac{1}{2}, \\
        \dfrac{s(n)}{l(n)}<\dfrac{2}{3}.
    \end{split}
    \end{align}

    The next lemma is a direct consequence of Proposition \ref{interval_length_words} and was extracted from the proof of Proposition 25 in \cite{CKKK}. For the sake of completeness we will replicate the argument.

    \begin{lemma}\label{connecting_words}
        Let $\al \in \N$ be such that $0<\al \leq l(n)$ and $b_1,b_2 \in \SB_n$. Then one of the following holds:
        \begin{enumerate}[label=(\arabic*)]
            \item There exist $c_1,c_2 \in \SB_n$ such that $|c_1c_2|=|b_1b_2| + \al$;
            \item There exist $c_1,c_2,c_3 \in \SB_n$ such that $|c_1c_2c_3|=|b_1b_2| + \al$.
        \end{enumerate}
    \end{lemma}

    \begin{proof}
        We first observe that $2 s(n) < |b_1b_2| + \al \leq 3l(n)$. If $2s(n) < |b_1b_2| + \al \leq 2l(n)$, then, by Proposition \ref{interval_length_words}, we can find $c_1,c_2 \in \SB_n$ such that $|c_1c_2| = |b_1b_2| + \al$. 
        
        It follows from \eqref{sn-ln-ratio} that $3s(n) < 2 l(n)$. As a consequence, if $|b_1b_2| + \al > 2l(n)$, then $3 s(n) < |b_1b_2| + \al \leq 3l(n)$ and, by Proposition \ref{interval_length_words}, we can find $c_1,c_2,c_3 \in \SB_n$ such that $|b_1b_2| + \al = |c_1c_2c_3|$. 
    \end{proof}

    In \cite{CKKK}, it is proved that under certain conditions on the sequence $\{t(n)\}_{n=1}^\infty \subset \N$, the obtained shift space $X$ possesses the dynamical properties, as stated below.
    
    \begin{theorem}[Theorem 27 in \cite{CKKK}]\label{exist_minimal_example}
        There exist sequences of positive integers $\{t(n)\}_{n=1}^\infty$ such that $X$ is minimal and mixing, and has positive topological entropy and the $\dbar$-shadowing property.
    \end{theorem}

    The sequences $\{t(n)\}_{n=1}^\infty$ mentioned in Theorem \ref{exist_minimal_example} fulfill the following conditions for every $n \in \N$:
    \begin{enumerate} [label=(\arabic*)]
        \item\label{cond_1} $t(n)> \frac{|\tau(n)|}{2l(n)-3 s(n)}$;
        \item\label{cond_2} $t(n)  \geq \frac{l(n)}{l(n)-s(n)}$;
        \item\label{cond_3} $t(n)  \geq \frac{2s(n)+2l(n) + 3|\tau(n)|}{l(n)}$;
        \item\label{cond_4} $t(n) > \frac{3 l(n) + |\tau(n)|}{s(n)2^{-n}}$.
    \end{enumerate}

    From now on, we will assume that the sequence $\{t(n)\}_{n=1}^\infty$ satisfies the conditions above. We will show that under these conditions the shift space $X$ is $\dbar$-complete. 
    
    \begin{lemma}\label{lemma_word_Bn_close}
        For every $x \in X$ and $n \in \N$ there exists a word $w \in \SB_{n+1}$ such that
        $$| \{ i \in [0, |w|) : x_i \neq w_i \} | \leq  3l(n) + |\tau(n)|.$$
    \end{lemma}

    \begin{proof}
        Let $x \in X$. Since $x \in X_n$, there exists $\al \in [0, l(n))$ such that $x_{[\al, \infty)} $ can be written as a concatenation of words from $\SB_n$, that is, there exists $\{b_i\}_{i=1}^\infty \subset \SB_n$ such that
        $$x_{[\al, \infty)} = b_1b_2\cdots .$$
        By Lemma \ref{connecting_words}, one of the following holds:
        \begin{enumerate}[label=(\arabic*)]
            \item there exist $c_1,c_2 \in \SB_n$ such that $|x_{[0,\al)} b_1b_2| = |c_1c_2|$;
            \item there exist $c_1,c_2, c_3 \in \SB_n$ such that $|x_{[0,\al)} b_1b_2| = |c_1c_2c_3|.$
        \end{enumerate}

        Let $w \in \SB_{n+1}$ be defined as
        \begin{equation*}
        w=
            \begin{cases}
            c_1c_2b_3 \ldots b_{t(n)} \tau(n), & \text{ if } (1) \text{ holds,}\\
            c_1c_2c_3b_3 \ldots b_{t(n)-1} \tau(n), & \text{ otherwise.}
            \end{cases}         
        \end{equation*}

        By construction, in the worst case scenario we would have 
       \[ | \{ i \in [0, |w|) : x_i \neq w_i \} | \leq |uc_1c_2c_3| + |\tau(n)| \leq 3 l(n) + |\tau(n)|. \qedhere\]
    \end{proof}

    We can improve Lemma \ref{lemma_word_Bn_close} to ensure that the prefix of the obtained word $w$ can be predetermined, while also providing more precise estimates on the number of positions at which subwords of $x_{[0,|w|)}$ and $w$ differ.
    
    \begin{lemma}\label{lemma_extension_word} 
        Let $x \in X$ and $n \in \N \backslash \{1\}$. For every $u \in \SB_{n}$ there exists $w \in \SB_{n+1}$ so that $u$ is prefix of $w$, and if $\ell \in [|u|, |w|)$, then
        $$| \{ i \in [|u|, \ell) : x_i \neq w_i \} | \leq  \ell \cdot 2^{-n+2}.$$
    \end{lemma}
    \begin{proof}
        Fix $u \in \SB_{n}$. Since $x_{[|u|, \infty)} \in X$, by Lemma \ref{lemma_word_Bn_close}, there exists $v^{(2)} \in \SB_{n}$ such that
        \begin{equation*}
            \left|\left\{i \in [0, |v^{(2)}|) : x_{|u| + i} \neq v^{(2)}_i\right\}\right| \leq 3 l(n-1) + |\tau(n-1)|.
        \end{equation*}

        Similarly, for each $j \in [3, t(n)]$ we can find $v^{(j)} \in \SB_{n}$ so that
        \begin{equation*}
            \left|\left\{i \in [0, |v^{(j)}|) : x_{|uv^{(2)} \ldots v^{(j-1)}| + i} \neq v^{(j)}_i\right\}\right| \leq 3 l(n-1) + |\tau(n-1)|.
        \end{equation*}

        Set $w = uv^{(2)} \ldots v^{(t(n))} \tau(n) \in \SB_{n+1}$. For each $\ell \in [|u|, |w| - |\tau(n)|)$ we take the largest $j \in [2, t(n))$ so that $\ell \geq |uv^{(2)} \ldots v^{(t(j))}|$ and, by construction, we have
        \begin{equation}\label{eq1_lemma_extension}
            \left|\left\{i \in [|u|, \ell) : x_i \neq w_i\right\}\right| \leq j \cdot \Bigl( 3l\left(n-1\right) + |\tau(n-1)| \Bigr) .
        \end{equation}

        Note that condition \ref{cond_4} implies $3l(n-1) + |\tau(n-1)| \leq t(n-1)s(n-1) \cdot 2^{-n+1}$. Since $\ell \geq j \cdot s(n)$, it follows from \eqref{eq1_lemma_extension} that
        \begin{align}
            \begin{split}
            |\{i \in [|u|, \ell) : x_i \neq w_i\}| &\leq j \cdot t(n-1)s(n-1)\cdot 2^{-n+1} \\
            &\leq j \cdot s(n) \cdot 2^{-n+1} \\
            &\leq  \ell \cdot 2^{-n+1}.
            \end{split}
        \end{align}

        For $\ell \in [|w|- |\tau|, |w|)$, note that condition \ref{cond_4} implies $|\tau(n)| < t(n)s(n)\cdot 2^{-n}$. Therefore, since $\ell \geq t(n)s(n)$, we conclude that
        \begin{align*}
            \begin{split}
            |\{i \in [|u|, \ell) : x_i \neq w_i\}| &\leq  t(n)s(n) \cdot 2^{-n+1} + |\tau(n)| \\
            &\leq t(n)s(n) \cdot 2^{-n+1}  + t(n)s(n)\cdot 2^{-n}\\
            &\leq  \ell \cdot 2^{-n+2}. \qedhere
            \end{split}
        \end{align*}
    \end{proof}

    Using a simple inductive argument we can extend Lemma \ref{lemma_extension_word} to the following result.

    \begin{lemma}\label{lemma_extension_word_ind} 
        Let $x \in X$ and $n \in \N$. For every $u \in \SB_{n}$ and $m > n$, there exists $w \in \SB_{m}$ so that $u$ is prefix of $w$ and if $\ell \in [|u|, |w|)$, then
        \begin{equation*}
            | \{ i \in [|u|, \ell) : x_i \neq w_i \} | \leq  \ell \cdot 2^{-n+2}.
        \end{equation*}
    \end{lemma}
    
    \begin{theorem}\label{thm_minimal_complete}
        The shift space $X$ is $\dbar$-complete.
    \end{theorem}

    \begin{proof}
        Let $\{x^{(n)}\}_{n=1}^\infty \subset X$ be a $\dbar$-Cauchy sequence. Passing, if necessary, to a subsequence we can assume that for every $n \in \N$ the following holds
        \begin{equation*}
            \dbar \left(x^{(n)}, x^{(n+1)}\right) < 2^{-n-2}.
        \end{equation*}

        There exists a strictly increasing sequence $\{k_n\}_{n=1}^\infty \subset \N$ such that for every $n \in \N$ and $\ell \geq k_n$ we have
          \begin{equation}\label{Cauchy_bound}
            \left| \left\{ i \in [0, \ell) : x^{(n)}_i \neq x^{(n+1)}_i \right\} \right| < \ell \cdot 2^{-n-2}.
        \end{equation}

        By the triangle inequality, it follows from \eqref{Cauchy_bound} that for every $m > n$ and $\ell \geq k_{m-1}$ we have 
        \begin{equation}\label{upper_bound_words}
            \left| \left\{  i\in [0,\ell) : x_i^{(n)}\neq x_i^{(m)} \right\}\right| \leq \ell \cdot \left(2^{-n-2}+ \ldots +2^{-m}\right) <\ell \cdot 2^{-n-1}.
        \end{equation}

        We will construct a strictly increasing sequence $\{m_j\}_{j=1}^\infty \subset \N$ and a sequence of words $\{w^{(j)}\}_{j = 1}^\infty \subset \lang(X)$ such that for all $j \in \N$ we have $w^{(j)} \in \SB_{m_j}$, and $w^{(j)}$ is a prefix of $w^{(j+1)}$. We will denote $\al_j = |w^{(j)}|$. We then define $w \in \alf^\infty$ to be the unique point such that for every $j\in \N$ we have $w_{[0, \al_j)} = w^{(j)}$. 

        Let $m_1 \geq 4$ be such that $s(m_1) > k_1$. By Lemma \ref{lemma_word_Bn_close}, there exists $w^{(1)} \in \SB_{m_1}$ such that 
        $$\left|\left\{ i \in [0,\al_1) : x^{(1)}_i \neq w^{(1)}_i \right\}\right| \leq 3l(m_1 -1) + \left|\tau(m_1 -1)\right|.$$

        The construction follows by taking $m_{j+1} > m_{j}$ satisfying $s(m_{j+1}) > k_{j+1}$. Note that $m_1 \geq 4$ implies $m_j \geq j+3$ for all $j \in \N$. By Lemma \ref{lemma_extension_word}, there exists $w^{(j+1)} \in \SB_{m_{j+1}}$ with $w^{(j)}$ as a prefix, and so that for every $\ell \in [\al_{j}, \al_{j+1})$ it holds that
        \begin{equation}\label{eq_alj_alj+1}
            \left|\left\{ i \in [\al_j, \ell) : x^{(j+1)}_i \neq w^{(j+1)}_i\right\}\right| \leq  \ell \cdot 2^{-m_{j} +2} \leq \ell \cdot 2^{-j-1}.
        \end{equation}
        
        Note that, since $\{m_j\}_{j=1}^\infty \nearrow \infty$ and $w_{[0, \al_j)} \in \SB_{m_j}$, we have $w \in X$. It remains to show that $\lim_{n \to \infty} \dbar(x^{(n)}, w) = 0$. Fix $n \in \N$. For every $\ell \geq \al_{n+1}$ there exists $k = k(\ell) \in \N$ such that $\ell \in [\al_{n+k}, \al_{n+k+1})$. By the triangle inequality, we have
        \begin{align*}
            \left| \left\{ i \in [0, \ell) : x^{(n)}_i \neq w_i \right\} \right| &\leq \left| \left\{ i \in [0, \al_n) : x^{(n)}_i \neq w_i \right\} \right| \\  
            & +   \sum_{j=0}^{k-1} \left| \left\{ i \in [\al_{n+j}, \al_{n+j+1}) : x^{(n)}_i \neq x^{(n+j+1)}_i \right\} \right| \\ &+
             \sum_{j=0}^{k-1} \left| \left\{ i \in [\al_{n+j}, \al_{n+j+1}) : x^{(n+j+1)}_i \neq w_i \right\} \right| \\ &+ 
             \left| \left\{ i \in [\al_{n+k}, \ell) : x^{(n)}_i \neq x^{(n+k+1)}_i \right\} \right| \\& +
              \left| \left\{ i \in [\al_{n+k}, \ell) : x^{(n+k+1)}_i \neq w_i \right\} \right|.
        \end{align*}

        It is not hard to see that $\al_{n+j+1} \geq 2 \al_{n+j}$ for all $j \in \N$. Moreover, by construction, we have $\al_{n+j} > k_{n+j}$ for every $j \in \N$.  Thus, using \eqref{upper_bound_words} we obtain
        \begin{equation}\label{eq_sum_interpolation_1}
            \sum_{j=0}^{k-1} \left| \left\{ i \in [\al_{n+j}, \al_{n+j+1}) : x^{(n)}_i \neq x^{(n+j+1)}_i \right\} \right| \leq \sum_{j=1}^{k} \al_{n+j} \cdot 2^{-n-1} \leq \al_{n+k} \cdot  2^{-n} .
        \end{equation}

        It also follows from \eqref{upper_bound_words} that
        \begin{equation}\label{last_interpolation}
            \left| \left\{ i \in [\al_{n+k}, \ell) : x^{(n)}_i \neq x^{(n+k+1)}_i \right\} \right| \leq \ell \cdot 2^{-n-1}.
        \end{equation}

        Using \eqref{eq_alj_alj+1} we obtain 
        \begin{align}\label{eq_suminterpolation_2}
        \begin{split}
            \sum_{j=0}^{k-1} \left| \left\{ i \in [\al_{n+j}, \al_{n+j+1}) : x^{(n+j+1)}_i \neq w_i \right\} \right| &\leq \sum_{j=1}^{k} \al_{n+j} \cdot 2^{-n-j-1} \\ &\leq \al_{n+k} \cdot 2^{-n}.
        \end{split}
        \end{align} 

        Hence, combining \eqref{eq_alj_alj+1}, \eqref{eq_sum_interpolation_1}, \eqref{last_interpolation} and \eqref{eq_suminterpolation_2}, we conclude that
        \begin{align*}
        \begin{split}
            \left| \left\{ i \in [0, \ell) : x^{(n)}_i \neq w_i \right\} \right| &\leq \al_n + \al_{n+k} \cdot 2^{-n} + \al_{n+k} \cdot 2^{-n} + \ell \cdot 2^{-n-1} + \ell \cdot 2^{-n-k-1} \\ 
            &\leq \al_n + \ell \cdot 2^{-n+2}.
        \end{split}
        \end{align*}

        Therefore, we conclude that $\dbar \left(x^{(n)}, w\right) \leq 2^{-n+2}$ for all $n \in \N$, and consequently, 
        $\lim_{n \to \infty} \dbar\left(x^{(n)}, w\right) = 0$.
        \end{proof}

    \begin{corollary}\label{minimal_AASP}
        The minimal shift space $X$ has the vague specification property.
    \end{corollary}
    
    \section{Applications}\label{section_applications}

    The equivalence of the vague specification and the asymptotic average shadowing properties allows us to use results already known about one property when dealing with the other. Consequently, we can address questions raised about systems that exhibit these properties.

    Kamae showed in \cite{Kamae} that the specification property together with expansivity are sufficient conditions for the vague specification property. However, it is already known that the specification property alone implies the asymptotic average shadowing property (see \cite{KKO,
    KLO, KLO2}). Therefore, we can remove the assumption of expansivity from Kamae's result. 

    \begin{corollary}
        If a topological dynamical system has the specification property, then it has the vague specification property.
    \end{corollary}

    The rest of the section is devoted to answering questions from \cite{DW} and \cite{KKO}.

    \begin{question}[\cite{DW}]\label{question_DW}
        What is the relation between the vague and weak specification properties?
    \end{question}

    \begin{question}[Question 10.2 in \cite{KKO}]\label{question_minimal_aasp}
        Is there a nontrivial minimal topological dynamical system with the almost specification property or the (asymptotic) average shadowing property?
    \end{question}

    \begin{question}[Question 10.3 in \cite{KKO}]\label{question_ASP_implies_AASP}
        Does the (asymptotic) average shadowing property imply the almost specification property? Does the average shadowing property imply the asymptotic average shadowing property?
    \end{question}

        \begin{question}[Question 10.6 in \cite{KKO}]\label{question_factors}
        Is the (asymptotic) average shadowing property or the almost specification property inherited by factors?
    \end{question}

    \subsubsection*{Answer to Question \ref{question_DW}}

    To answer this question, we rely on results from \cite{KKO} and \cite{KLO2}. We decided to skip the definitions of the weak specification and the almost specification properties, and only invoke the results that are relevant to our purpose. We refer to \cite{KLO} for definitions and further information. 

    \begin{theorem}[Theorem 18 in \cite{KLO2}]
        If a topological dynamical system has the weak specification property, then it has the asymptotic average shadowing property.
    \end{theorem}

    Immediately from Theorem \ref{thm:AASP-VSP} we obtain the following corollary.
    \begin{corollary}\label{weak_spec_implies_vague_spec}
        If a topological dynamical system has the weak specification property, then it has the vague specification property.
    \end{corollary}

    The converse is not true, and the almost specification property is used to prove it. Due to \cite{KKO}, it is known that the asymptotic average shadowing property (and consequently the vague specification property) is a consequence of the almost specification property. 

    \begin{theorem}[Theorem 3.5 in \cite{KKO}]
        If a topological dynamical system has the almost specification property, then it has the asymptotic average shadowing property.
    \end{theorem}

    \begin{corollary}\label{cor:almostimpvag}
        If a topological dynamical system has the almost specification property, then it has the vague specification property.
    \end{corollary}

    It is known that the weak specification and the almost specification properties are independent of each other. Indeed, Theorem 17 in \cite{KOR} shows that there are topological dynamical systems with almost specification but without weak specification. 
    In Section 6 of \cite{KOR} there is a construction of a topological dynamical system with weak specification, but without almost specification (we also refer to \cite{KKO, KLO, KOR,Pavlov} for more details). 
    
    Therefore, Corollary \ref{cor:almostimpvag} implies the following result, which together with Corollary \ref{weak_spec_implies_vague_spec}, answers Question \ref{question_DW}. 
    
    \begin{corollary}
        There are topological dynamical systems with the vague specification property but without the weak specification property. 
    \end{corollary}

    \subsubsection*{Answer to Question \ref{question_minimal_aasp}}

    In Section \ref{minimal_example} we present the construction of a class of minimal shift spaces with the $\dbar$-shadowing property from \cite{CKKK} (in particular, these shift spaces have the average shadowing property). By Corollary \ref{minimal_AASP}, every shift space in this class also has the asymptotic average shadowing property, which shows that Question \ref{question_minimal_aasp} has a positive answer.
    
    \subsubsection*{Answer to Question \ref{question_ASP_implies_AASP}}

   The relation between the asymptotic average shadowing property and the almost specification is already clarified in \cite{KLO}. 
   
    A step toward an answer for the second part of Question \ref{question_ASP_implies_AASP} is given by Theorem \ref{equiv_AASP_ASP+Besicovitch} for surjective topological dynamical systems and by Corollary \ref{cor:AASP_equiv_dbar-shadow} for surjective shift spaces. Although we were unable to provide a definitive answer to the question, the lack of examples with the average shadowing property but not Besicovitch complete suggests a potential positive answer. 
    
    \subsubsection*{Answer to the Question \ref{question_factors}}

    Theorem \ref{thm:AASP-VSP} also provides a partial answer to this question. Indeed, we can use Proposition 2 in \cite{Kamae}, which says that the vague specification property is inherited by factors, to confirm that the asymptotic average shadowing property is inherited by factors.

    \begin{corollary}
        Every factor of a topological dynamical system with the asymptotic average shadowing property has the asymptotic average shadowing property.
    \end{corollary}

    \section*{Acknowledgements}
 The authors wish to express their gratitude to Dominik Kwietniak for suggesting the problem and providing insightful discussions. We also thank the anonymous reviewers for their careful reading and valuable comments that helped us improve this work.

    \bibliographystyle{alpha}
    \bibliography{bibliography.bib}
\end{document}